\documentclass[final]{jcmlatex_no_title}

\usepackage{fancyhdr}
\usepackage{extramarks}
\usepackage{amsmath}
\usepackage{amssymb}
\usepackage{stmaryrd}
\usepackage{amsfonts}
\usepackage{indentfirst} 
\usepackage{tikz}
\usepackage[plain]{algorithm}
\usepackage{algpseudocode}
\usepackage{multirow}
 \usepackage{indentfirst} 
\usepackage{epstopdf}
\usepackage{listings}
\usepackage{color}
\usepackage{subfig}
\usepackage{bm}

\definecolor{dkgreen}{rgb}{0,0.6,0}
\definecolor{gray}{rgb}{0.5,0.5,0.5}
\definecolor{mauve}{rgb}{0.58,0,0.82}
\usetikzlibrary{automata,positioning}

\def\JCMvol{??}
\def\JCMno{??}
\def\JCMyear{????}
\def\JCMreceived{xxxx}
\def\JCMrevised{xxxx}
\def\JCMaccepted{xxxx}
\numberwithin{equation}{section} \numberwithin{figure}{section}

\renewcommand{\part}[1]{\textbf{\large Part \Alph{partCounter}}\stepcounter{partCounter}\\}

\begin{document}

\markboth{X. JIN AND S. WU} {$n$-rectangle nonconforming $H^3$ elements}
         
\title{TWO FAMILIES OF $n$-RECTANGLE NONCONFORMING FINITE ELEMENTS FOR SIXTH-ORDER ELLIPTIC EQUATIONS
\thanks{The work of Shuonan Wu is supported in part by the National Natural Science Foundation of China grant No. 12222101.}}

\author{Xianlin Jin
\thanks{School of Mathematical Sciences, Peking University, Beijing 100871, China \\ 
Email: xianlincn@pku.edu.cn} \and
Shuonan Wu
\thanks{School of Mathematical Sciences, Peking University, Beijing 100871, China \\
Email: snwu@math.pku.edu.cn} }

\maketitle

\begin{abstract}
In this paper, we propose two families of nonconforming finite
  elements on $n$-rectangle meshes of any dimension to solve the
  sixth-order elliptic equations.  The unisolvent property and the
  approximation ability of the new finite element spaces are
  established. A new mechanism, called \textit{the exchange of
  sub-rectangles}, for investigating the weak continuities of the proposed elements is discovered. With the help of some conforming relatives for the $H^3$ problems, we establish the quasi-optimal error estimate for the tri-harmonic equation in the broken $H^3$ norm of any dimension. The theoretical results are validated further by the numerical tests in both 2D and 3D situations.
\end{abstract}

\begin{classification}
65N30.
\end{classification}

\begin{keywords}
nonconforming finite element method, $n$-rectangle element,
  sixth-order elliptic equation, exchange of sub-rectangles
\end{keywords}

\section{Introduction} \label{sec:intro}

Sixth-order partial differential equations have been widely used to model various physical laws and dynamics in material sciences and phase field problems \cite{caginalp1986higher, elder2002modeling}. 
Owning such a significance in these areas, however, methods for solving the sixth-order equations are insufficient and less studied compared with the lower-order equations from both theoretical and numerical aspects. 
From a practical point of view, nonconforming finite element method is one of the frequently desired numerical methods for high order partial differential equations. 
In terms of solving sixth-order equations, the usage of nonconforming spaces allows us to avoid the requirement of $C^2$-continuity which causes high complexity for the implementation. 
Having a smaller set of degrees of freedom (DoFs) and a shrunken space of shape functions, yet the nonconforming finite elements should conceivably possess some basic weak continuity properties \cite{wang2001on} to preserve the convergence of the numerical solutions. 
Therefore, the design of such exquisite finite element spaces can be challenging for certain problems, especially in high dimensional situations.

Starting from the solving of fourth-order equations, there are several well-known nonconforming finite elements like the Morley element and the Zienkiewicz element designed on two-dimensional simplicial meshes. 
A similar idea was then applied to high dimensional case \cite{wang2007new} which generalizes the Zienkiewicz element to $n$-dimensional simplexes where $n\geq 2$. 
Further in \cite{wang2013minimal}, Wang and Xu proposed a family of nonconforming finite elements on simplexes named by the Morley-Wang-Xu element to solve $2m$-th-order elliptic equations where $m \leq n$. This result has been extended to $m=n+1$ in \cite{wu2019nonconforming}, and to arbitrary $m,n$ with interior stabilization \cite{wu2017interior}.  Restricted to the two-dimensional case, the nonconforming finite element spaces for $H^3$ or higher regularity can be found in \cite{hu2017canonical, hu2019cubic}.

On the simplicial meshes, other types of discretization besides the nonconforming finite element method for sixth order partial differential equations may also be feasible. 
In two-dimensional case, the $H^3$ conforming finite element was constructed in
\cite{vzenivsek1970interpolation} and can be generalized to arbitrary $H^m$ \cite{bramble1970triangular}. Recently, a construction of conforming finite element spaces with arbitrary smoothness in any dimension was given in \cite{hu2021construction}. Others include mixed methods \cite{schedensack2016new, droniou2019mixed}, $C^0$ interior penalty discontinuous Galerkin method \cite{gudi2011interior}, recovery-based method \cite{guo2018linear}, and virtual element methods \cite{chen2020nonconforming, chen2013nonconforming}.

As for rectangle meshes, successful constructions of finite element such as the Adini element \cite{adini1960analysis} of $C^0$ smoothness and Bogner-Fox-Schmidt element (BFS, \cite{bogner1965generation}) of $C^1$ smoothness were made on two-dimensional grids, whose DoFs are all defined on vertices of rectangles. 
After an extension \cite{wang2007some} to the $n$-rectangle meshes of any high dimensional spaces where $n\geq 2$, the Adini element and the BFS element possess only $C^0$ smoothness, and yet their solvabilities to the fourth-order equations have both been remained. Furthermore, an extended version of the Morley element to the $n$-rectangle meshes was also reported in \cite{wang2007some}.
 For the biharmonic equation,  a new family of $n$-rectangle nonconforming finite element by enriching the second-order serendipity element was constructed in \cite{zhou2020high}. 
 For arbitrary smoothness, a family of minimal $n$-rectangle macro-elements was established in \cite{hu2015minimal}.  
 
 
  In \cite{wang2007some}, Wang, Shi and Xu showed that the Morley, Adini and BFS element are of the first-order convergence in the energy norm for solving the biharmonic equation. A more delicate analysis proposed in \cite{hu2016capacity} reveals that the Adini element is capable of reaching a second-order convergence in the energy norm and has an optimal second-order convergence in the $L^2$-norm. 
It cannot be overlooked that theories of nonconforming finite element methods are well-prepared for the fourth-order equations on a variety of $n$-rectangle discretizations, yet very little is extended to the solving of sixth-order problems. 

In this paper, we develop two families of $n$-rectangle nonconforming finite elements for sixth-order partial differential equations. Both the two families of elements are constructed by enriching the DoFs of the $n$-rectangle Adini element \cite{wang2007some} and the corresponding shape function space. Following the well-developed projection-averaging strategy \cite{wang2007some}, we give the definition of the interpolation operator in high dimensional cases for both two families of elements. It can be shown that the shape function spaces are capable of approximating $H^{3+s}(\Omega)$ for any $s\in [0,1]$ in an arbitrarily high dimension, which are essential to the error estimate afterwards. 

Furthermore, analysis of the weak continuity properties usually plays
an important role in the investigation of a nonconforming finite
element. Reasonably, difficulties brought by the sixth-order
differential operator $(-\Delta)^3$ mainly occur when considering the
weak continuities of the following second-order derivatives of the
finite element function: the tangential-tangential
($\partial_{\tau\tau}$), normal-normal ($\partial_{\nu\nu}$) and
tangential-normal ($\partial_{\tau\nu}$) derivatives across the
$(n-1)$-dimensional faces of an element $T$. It is possible to make
use of the interpolations of other well-known $n$-rectangle finite
elements to locally estimate the terms of $\partial_{\tau\tau}$ and
$\partial_{\nu\nu}$. However, the analysis of $\partial_{\tau\nu}$ is
way more complicated for both the two families of elements, so that we
only consider estimating this term in a more global manner. We
therefore propose a new technique called {\it exchange of
sub-rectangles} to deal with this complicated term. Combining the
results of weak continuities and the help of conforming relatives, we
complete estimating the consistency error, which gives the final error
estimate by applying the well-known Strang's Lemma.

%

Given a multi-index $\alpha = (\alpha_1,\alpha_2,\cdots,\alpha_n)$, we set $|\alpha| = \sum_{i=1}^{n}\alpha_i$ and $x^\alpha = x_1^{\alpha_1}x_2^{\alpha}\cdots x_n^{\alpha_n}$ for $x\in \mathbb{R}^n$. For a subset $B \subset \mathbb{R}^n$ and a nonnegative integer $r$, let $\mathcal{P}_r(B)$ and $Q_r(B)$ be the spaces of polynomial on $B$ defined by 
$$
\mathcal{P}_r(B) := \operatorname{span}\{x^\alpha~|~ |\alpha| \leq r \}, \qquad 
Q_r(B) := \operatorname{span}\{x^\alpha~|~ \alpha_i \leq r \}.
$$ 
Moreover, we denote $Q_1^{\hat{i}}(B)$ as the subspace of $Q_1(B)$ with no dependence on $x_i$, i.e., 
\begin{equation} \label{eq:Q1-hat-i}
Q_1^{\hat{i}}(B) := \operatorname{span} \{x^\alpha~|~ \alpha_i = 0, \alpha_j \leq 1\}.
\end{equation}
For any finite dimensional sets of functions $A$ and $B$, we denote by $A\cdot B := \operatorname{span} \{ab ~|~ a\in A, b \in B\}$.
In this paper, we will also use the notation $x \lesssim y$ to represent $x \leq C y$ for some constant $C$ independent of the crucial parameter such as the mesh size $h$.

The rest of the paper is organized as follows. 
In Section \ref{sec:element} we introduce some basic notations and give definitions to the two families of $n$-rectangle nonconforming finite element. Unisolvent properties and part of the weak continuities are also developed herein. 
The approximation properties of the nonconforming spaces are discussed and proved in Section \ref{sc:approximation}, where same methods are used to verify the existence of some necessary conforming relatives. 
In Section \ref{sc:tn} we present the main technique of analyzing the weak continuity of $\partial_{\tau\nu}$ derivatives and several attached conclusions. 
Finally we give the full estimate of the numerical solutions of our new finite elements in Section \ref{sc:convergence} and three numerical examples to verify our theories in Section \ref{sc:numerical}. 
Concluding remarks are given in Section \ref{sc:conclusion}.

\section{The $H^3$-nonconforming $n$-Rectangle Elements} \label{sec:element}

In this section, we construct two families of $H^3$-nonconforming elements which are defined on the $n$-rectangle meshes. Let $\Omega \subset \mathbb{R}^n~(n\geq2)$ denote a bounded polyhedral domain with boundary $\partial \Omega$, $\nu = \left(\nu_1, \nu_2, \cdots, \nu_n \right)^{\top}$ be the unit outer normal vector to $\partial \Omega$, and $\mathcal{T}_h$ be a quasi-uniform $n$-rectangle discretization on $\Omega$ with the mesh size $h>0$. 

Throughout this paper, we will use the standard notations of the Sobolev spaces. Let $m\geq 0$ be an integer, we define the following mesh-dependent norm and semi-norm: 
\begin{equation*}
\|v\|_{m,h} = \left( \sum_{T\in\mathcal{T}_h}\|v\|^2_{m,T} \right)^{1/2}, \quad |v|_{m,h} = \left( \sum_{T\in\mathcal{T}_h}|v|^2_{m,T} \right)^{1/2},
\end{equation*}
for a function $v$ with $v|_{T} \in H^m(T), ~\forall T\in \mathcal{T}_h$.

\subsection{Preliminaries}
For a given point $c = (c_1,c_2,\cdots, c_n)^\top \in \mathbb{R}^n$ and $h_1,h_2,\cdots,h_n$ being $n$ positive numbers, an $n$-rectangle $T$ is described in the barycentric coordinate $\xi = (\xi_1,\xi_2,\cdots, \xi_n)^\top $ as follows:
\begin{equation} \label{eq:n-rectangle}
T = \lbrace x\in \mathbb{R}^n ~|~ x_i = c_i + h_i \xi_i, ~ -1\leq \xi_i \leq 1, ~ 1\leq i \leq n \rbrace,
\end{equation}
with $2^n$ vertices given by 
$$
a_i := (c_1 + \xi_{i1}h_1, c_2 + \xi_{i2}h_2, \cdots, c_n + \xi_{in}h_n)^\top, \quad 1\leq i\leq 2^n.
$$
Here, the values $(\xi_{i1}, \xi_{i2}, \cdots \xi_{in})^\top =(\pm 1, \pm 1, \cdots, \pm 1)^\top$ for $1 \leq i \leq 2^n$. The $(n-1)$-dimensional faces of the element $T$ are denoted by
\begin{equation*}
F_{i}^{\pm} := \lbrace x\in \partial T~|~ \xi_i=\pm 1, -1\leq \xi_j \leq 1, 1\leq j\leq n, j\neq i \rbrace \quad 1\leq i\leq n,
\end{equation*}
whose barycenters are written as $b_i^\pm := (c_1, \cdots, c_{i-1}, c_i \pm h_i, c_{i+1}, \cdots, c_n)^\top$.

Following the standard description in \cite{brenner2008mathematical}, a finite element can be represented by a triple $(T,\mathcal{P}_T, \mathcal{N}_T)$ where $T$, taken as an $n$-rectangle \eqref{eq:n-rectangle}, describes the geometric shape, $\mathcal{P}_T$ the shape function space and $\mathcal{N}_T$ the vector of degrees of freedom (DoFs). We first review several $n$-rectangle finite elements that will be helpful for further analysis. 
 
 \begin{enumerate}
 \item $n$-rectangle $Q_1$ element: $\mathcal{P}_T := Q_1(T)$ and the DoFs are defined as 
 $$
 \mathcal{N}_T(v) = \left(v(a_1), v(a_2), \cdots, v(a_{2^n}) \right)^\top.
 $$ 
Further, it is well-known that the polynomials
\begin{equation}\label{eq: Q1-basis}
p_{0i} = \dfrac{1}{2^n} \prod_{j=1}^{n}(1+\xi_{ij} \xi_j),\quad 1\leq i\leq 2^n
\end{equation}
form a set of basis functions of the space $Q_1(T)$. Accordingly, the canonical interpolation operator $\Pi_T^{\bm{0}}: C^0(T) \rightarrow Q_1(T)$ is defined as
\begin{equation*} 
\mathcal{N}_T(\Pi_T^{\bm{0}} v) = \mathcal{N}_T(v), \quad \text{or} \quad \Pi_T^{\bm{0}} v := \sum_{i=1}^{2^n} p_{0i} v(a_i), \quad \forall v \in C^0(T).
\end{equation*}

\item $n$-rectangle Adini element \cite{wang2007some}: $\mathcal{P}_T := Q_1(T) \cdot \operatorname{span}\{1, x_i^2~|~ 1 \leq i \leq n\}$ and the DoFs are defined as 
$$
\mathcal{N}_T(v) = \left( v(a_1),\nabla v(a_1)^\top, v(a_2),\nabla v(a_2)^\top, \cdots, v(a_{2^n}),\nabla v(a_{2^n})^\top \right)^\top.
$$ 
The canonical interpolation operator is denoted by $\Pi_T^{\bm{1}}$.
\item $n$-rectangle partial Adini element: $\mathcal{P}_T := Q_1(T) \cdot \operatorname{span}\{1, x_i^2\}$, and the DoFs are defined as 
$$
\mathcal{N}_T(v) = \left( v(a_1), \frac{\partial v}{\partial x_i}(a_1), v(a_2), \frac{\partial v}{\partial x_i}(a_2), \cdots, v(a_{2^n}),\frac{\partial v}{\partial x_i}(a_{2^n}) \right)^\top.
$$ 
The canonical interpolation operator is denoted by  $\Pi_T^{\bm{e}_i}$.
 \end{enumerate}

For any $v$ in the finite element spaces by the above elements, on any $(n-1)$-dimensional face $F$ of $T \in \mathcal{T}_h$, the restriction of $v|_F$ is a polynomial of $(n-1)$ variables in the shape function space $\mathcal{P}(F)$. Then $v|_F$ is uniquely determined by the DoFs on $F$ (which also proves the unisolvent properties of the above elements by induction on the dimension). Therefore, $v$ is continuous through $F$.  Next, for any piecewise smooth function $v$ with the same inter-element degrees of freedom, the interpolation operator can be given element by element, i.e.,
\begin{equation}\label{eq:global-Q1-Pi}
\Pi_h^{\bm{\beta}} \vert_{T} v := \Pi_T^{\bm{\beta}} v,\quad \forall T \in \mathcal{T}_h, \qquad \bm{\beta} = \bm{0}, \bm{e}_i,\text{or } \bm{1}. 
\end{equation}
Here, we unify the notations by denoting $\beta_i$ as the highest order of derivative along $x_i$.

\subsection{The $n$-rectangle Morley-type element}
Define
\begin{equation} \label{eq:Morley-shape}
\mathcal{P}_M(T) := Q_1(T) \cdot \operatorname{span}\{1, x_i^2 ~|~ 1\leq i \leq n\} + \operatorname{span}\{ x_i^4, x_i^5 ~|~ 1 \leq i \leq n\}.
\end{equation}
It can be verified that $\mathcal{P}_3(T) \subset \mathcal{P}_M(T)$. For the $n$-rectangle Morley-type element, $\mathcal{P}_T$ and $\mathcal{N}_T$ are given by (see Fig. \ref{fig:Morley-element}):

\begin{figure}[H]
\centering
\subfloat[Rectangular element.]
{
\includegraphics[width=1.8in]{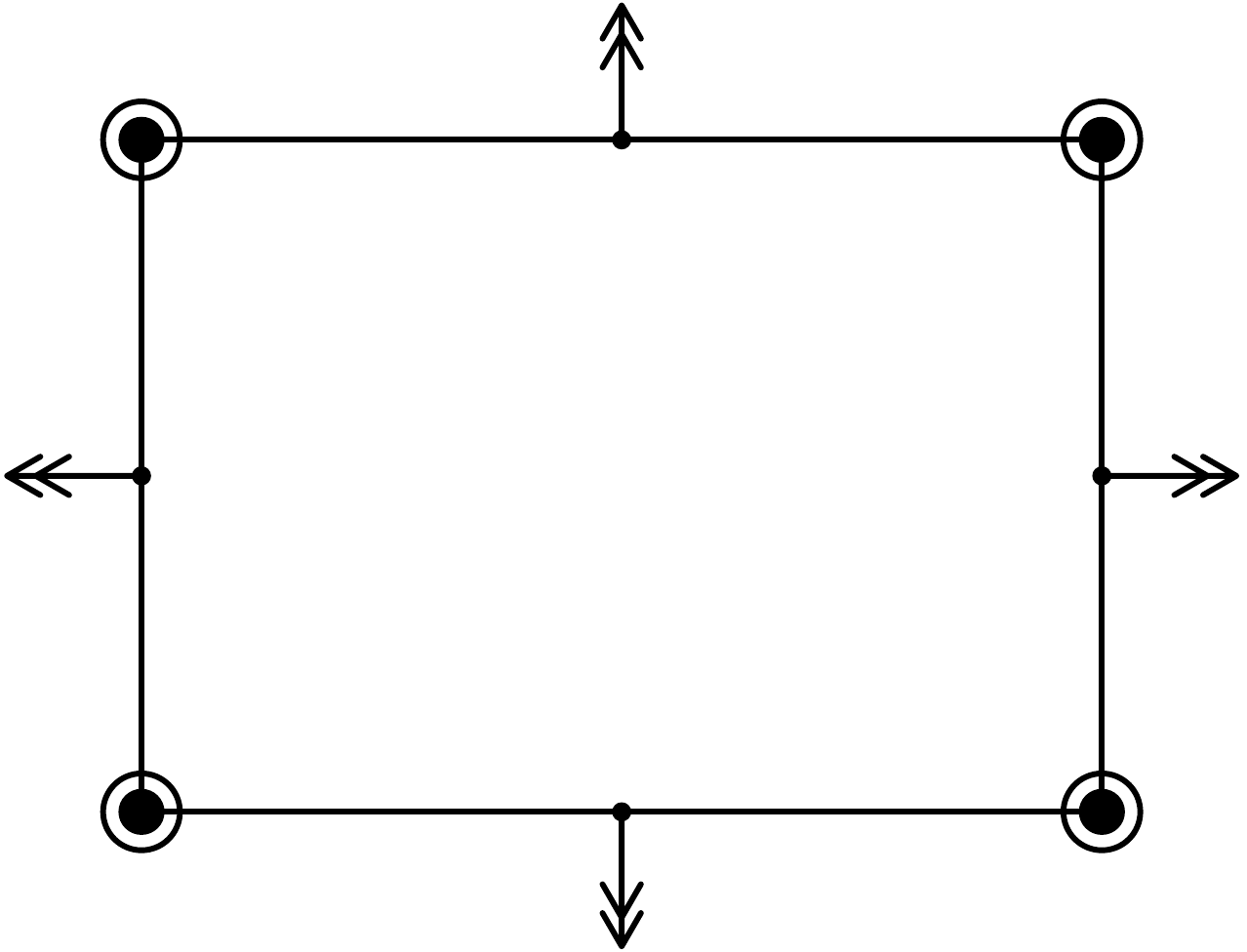} 
\label{fig:rect-Morley}
} \qquad 
\subfloat[Cubic element.]
{
\includegraphics[width=2.0in]{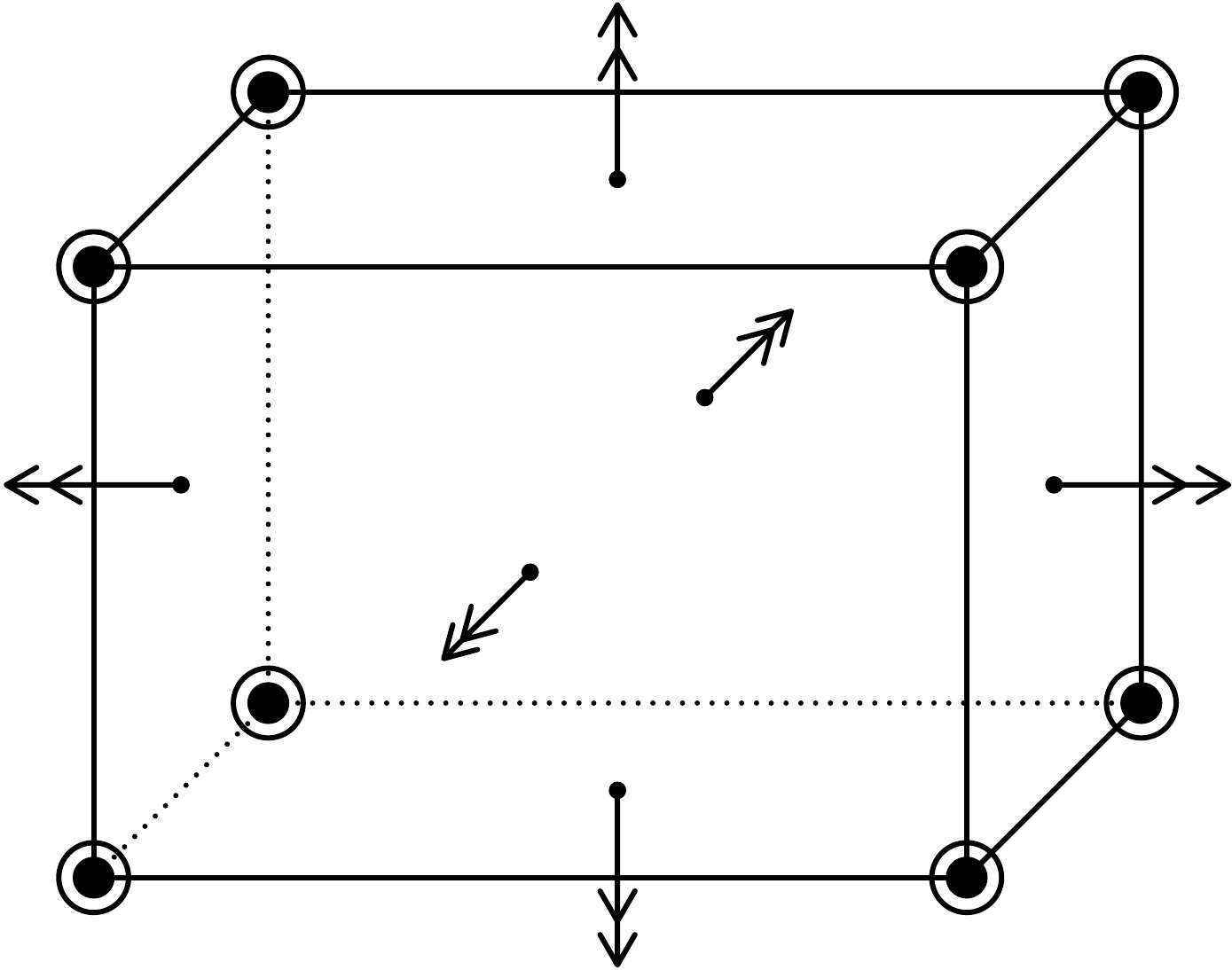} 
\label{fig:cubic-Morley}
} 
\caption{Degrees of freedom of the $H^3$-nonconforming Morley-type element.} \label{fig:Morley-element}.
\end{figure}

\begin{itemize}
\item $\mathcal{P}_T = \mathcal{P}_{M}(T)$.
\item For $v \in C^2(T)$, the vector $\mathcal{N}_T(v) $ of degree of freedom is
\begin{equation*}
\mathcal{N}_T(v) = \left( v(a_1),\nabla v(a_1)^\top, , \cdots, v(a_{2^n}),\nabla v(a_{2^n})^\top,  \frac{\partial^2 v}{\partial \nu^2}(b_1^\pm), \cdots, \frac{\partial^2 v}{\partial \nu^2}(b_{n}^\pm) \right)^\top.
\end{equation*}
\end{itemize}

The basis functions of the $n$-rectangle Morley element is denoted by $p_{0i}$ (i.e., corresponding to the nodal values), $p_{ji}$ (i.e., corresponding to $\frac{\partial v(a_i)}{\partial x_j}$), and $r_{k}^\pm$ (i.e., corresponding to the 2nd normal derivative on the face center $b_{k}^\pm$), which are given by 
\begin{equation}\label{eq:Morely-basis}
\left\{
\begin{aligned}
&p_{0i} = \dfrac{1}{2^{n+1}} \Big (2+\sum_{k=1}^n (\xi_{ik}\xi_{k}-\xi_k^2) \Big ) \prod_{k=1}^{n} (1+\xi_{ik}\xi_k) + \dfrac{3}{2^{n+3}} \sum_{k=1}^{n} \xi_{ik}\xi_k(\xi_k^2-1)^2, &1\leq i\leq 2^n,\\
&p_{ji} = \dfrac{h_j \xi_{ij}}{2^{n+1}}(\xi_j^2-1)\prod_{k=1}^{n}(1+\xi_{ik}\xi_k) - \dfrac{h_j}{2^{n+3}}(\xi_{ij}+3\xi_j)(\xi_j^2-1)^2, &1\leq i\leq 2^n, 1\leq j\leq n,\\ 
&r_{k}^\pm = \pm\dfrac{h_k^2}{16}(\xi_k+1)^2(\xi_k-1)^2(\xi_k \pm 1), &1\leq k\leq n.
\end{aligned}
\right.
\end{equation}

For the $n$-rectangle Morley-type element, we can define the corresponding $H^3$-nonconforming finite element spaces $V_h$ and $V_{h0}$ as follows: $V_h$ consists of all functions $v_h$ such that for any $T \in \mathcal{T}_h$: (1) $v_h|_T \in \mathcal{P}_M(T)$, (2) $v_h$ is $C^1$-continuous at all vertices of $T$, and (3) the second normal derivatives of $v_h$ is continuous at the barycenters of all $(n-1)$-dimensional faces of $T$; $V_{h0}$ consists of all functions $v_h \in V_h$ such that for any $T \in \mathcal{T}_h$, $v_h$ and $\nabla v_h$ vanish at the vertices of $T$ belonging to $\partial \Omega$ and the second normal derivative of $v_h$ vanishes at the barycenter of all $(n-1)$-dimensional faces of $T$ on $\partial \Omega$.

It can be seen that the DoFs for Morley-type finite element consists of that for Adini finite element space and the second-order normal derivative on faces.  Moreover, $\mathcal{P}_M(T)$ contains the shape function space of the Adini element. Therefore, 
\begin{equation} \label{eq:Morley-Adini}
\left(v_h - \Pi_h^{\bm{1}} v_h\right)|_T \in \operatorname{span}\{r_k^\pm~|~ 1\leq k \leq n\} \quad \forall v_h \in V_h. 
\end{equation}
Here, we recall that $\Pi_h^{\bm{1}}$ stands for the interpolation to Adini finite element space \eqref{eq:global-Q1-Pi}. 


\begin{lemma}[tangential-tangential weak continuity for Morley] \label{lm:Morley-tt}
Let $V_h$ and $V_{h0}$ be the finite element spaces of the $n$-rectangle Morley-type element. Then, 
\begin{equation} \label{eq:Morley-tt}
\int_F \dfrac{\partial^2}{\partial \tau_1 \partial\tau_2} (v|_T) = \int_F  \dfrac{\partial^2}{\partial \tau_1\partial \tau_2} (v|_{T'}) \quad \forall v \in V_h,
\end{equation}
where $T,T' \in \mathcal{T}_h$ share a common $(n-1)$-dimensional interior face $F$,  $\tau_1$ and $\tau_2$ are the unit tangential vectors on $F$. Moreover, if an $(n-1)$-dimensional face $F$ of $T \in \mathcal{T}_h$ is on $\partial \Omega$, then 
\begin{equation} \label{eq:Morley-tt-bc}
\int_F  \dfrac{\partial^2}{\partial \tau_1\partial \tau_2} (v|_T) = 0 \quad \forall v \in V_{h0}.
\end{equation}
\end{lemma}
\begin{proof}
We first observe that the basis function $r_k^\pm$ depends only on $\xi_k$ and vanishes on $F_k^\pm$. On any face $F_j^\pm (j \neq k)$, we have 
$$ 
\int_{F_j^\pm} \dfrac{\partial^2 r_k^\pm}{\partial x_k^2} = 
h_k^{-2} \int_{F_j^\pm} \dfrac{\partial^2 r_k^\pm}{\partial \xi_k^2} = 2^{n-2} h_k^{-2} |F_j^\pm| \left.\dfrac{\partial r_k^\pm}{\partial \xi_k}\right|_{\xi_k=-1}^{\xi_k=1} = 0.
$$ 
Using \eqref{eq:Morley-Adini} and the fact that the Adini finite element space is continuous \cite{wang2007some}, we have
$$
\int_F \dfrac{\partial^2}{\partial \tau_1 \partial\tau_2} (v|_T) - \int_F  \dfrac{\partial^2}{\partial \tau_1\partial \tau_2} (v|_{T'}) = 
\int_F \dfrac{\partial^2}{\partial \tau_1 \partial\tau_2} (v- \Pi_h^{\bm{1}} v|_T) - \int_F  \dfrac{\partial^2}{\partial \tau_1\partial \tau_2} (v - \Pi_h^{\bm{1}} v|_{T'}) = 0.
$$
This proves \eqref{eq:Morley-tt}. For $v \in V_{h0}$, we have $\Pi_h^{\bm{1}} v|_{\partial \Omega} = 0$, which leads to \eqref{eq:Morley-tt-bc}.  $\blacksquare$
\end{proof}


\begin{lemma}[normal-normal weak continuity for Morley] \label{lm:Morley-nn}
Let $V_h$ and $V_{h0}$ be the finite element spaces of the $n$-rectangle Morley-type element. Then, 
\begin{equation} \label{eq:Morley-nn}
\int_F \dfrac{\partial^2}{\partial \nu^2} (v|_T) = \int_F  \dfrac{\partial^2}{\partial \nu^2} (v|_{T'}) \quad \forall v \in V_h,
\end{equation}
where $T,T' \in \mathcal{T}_h$ share a common $(n-1)$-dimensional interior face $F$. Moreover, if an $(n-1)$-dimensional face $F$ of $T \in \mathcal{T}_h$ is on $\partial \Omega$, then 
\begin{equation} \label{eq:Morley-nn-bc}
\int_F  \dfrac{\partial^2}{\partial \nu^2} (v|_T) = 0 \quad \forall v \in V_{h0}.
\end{equation}
\end{lemma}
\begin{proof}
On any face $F_k^\pm$, we have 
\begin{equation}
\left.\dfrac{\partial^2 p_{0i}}{\partial \nu^2}\right|_{F_k^\pm} = \left.\dfrac{\partial^2 p_{0i}}{\partial x_k^2}\right|_{\xi_k=\pm 1} =  \mp\dfrac{3}{2^n}\xi_{ik}\prod_{j\neq k}(1+\xi_{ij}\xi_j) \pm \dfrac{3}{2^n} \xi_{ik},
\end{equation}
and for $p_{ji}$ with $1\leq j\leq n$,
\begin{equation}
\left.\dfrac{\partial^2 p_{ji}}{\partial \nu^2}\right|_{F_k^\pm} = \left.\dfrac{\partial^2 p_{ji}}{\partial x_k^2}\right|_{\xi_k=\pm 1} = \dfrac{1}{2^n}\left( \xi_{ik} \pm 3 \right)\prod_{j\neq k}(1+\xi_{ij}\xi_j) - \dfrac{1}{2^n}\left( \xi_{ik} \pm 3 \right).
\end{equation}
A straightforward computation gives 
\begin{equation}
\int_{F_{k}^{\pm}} \dfrac{\partial^2 p_{ji}}{\partial \nu^2} = 0,\quad 0\leq j\leq n.
\end{equation}
Moreover, we also have 
\begin{equation}
\int_{F_{k}^{+}} \dfrac{\partial^2 r_{j}^{+}}{\partial x_j^2} = 
\begin{cases}|F_{k}^{+}|, & j=k \\ 0, & \text { otherwise }\end{cases},\quad
\int_{F_{k}^{-}} \dfrac{\partial^2 r_{j}^{-}}{\partial x_j^2} = 
\begin{cases}|F_{k}^{-}|, & j=k \\ 0, & \text { otherwise }\end{cases},
\end{equation}
and 
\begin{equation}
\int_{F_{k}^{+}} \dfrac{\partial^2 r_{j}^{-}}{\partial x_j^2} = \int_{F_{k}^{-}} \dfrac{\partial^2 r_{j}^{+}}{\partial x_j^2} = 0,
\end{equation}
for all $1\leq j \leq n$. This gives the desired result. $\blacksquare$
\end{proof}

\subsection{The $n$-rectangle Adini-type element}
Define 
\begin{equation}\label{eq:shape-Adini}
\mathcal{P}_{A}(T) = Q_1(T)\cdot\operatorname{span}\lbrace 1, x_i^2, x_i^4~|~ 1 \leq i \leq n \rbrace.
\end{equation}
It is straightforward that $\mathcal{P}_3(T)\subset \mathcal{P}_A(T)$. The Adini-type element (see Fig.~\ref{fig:finite-element}) is then given by the triple $(T, \mathcal{P}_T, \mathcal{N}_T)$, where

\begin{figure}[H]
\centering
\subfloat[Rectangular element.]
{
\includegraphics[width=1.8in]{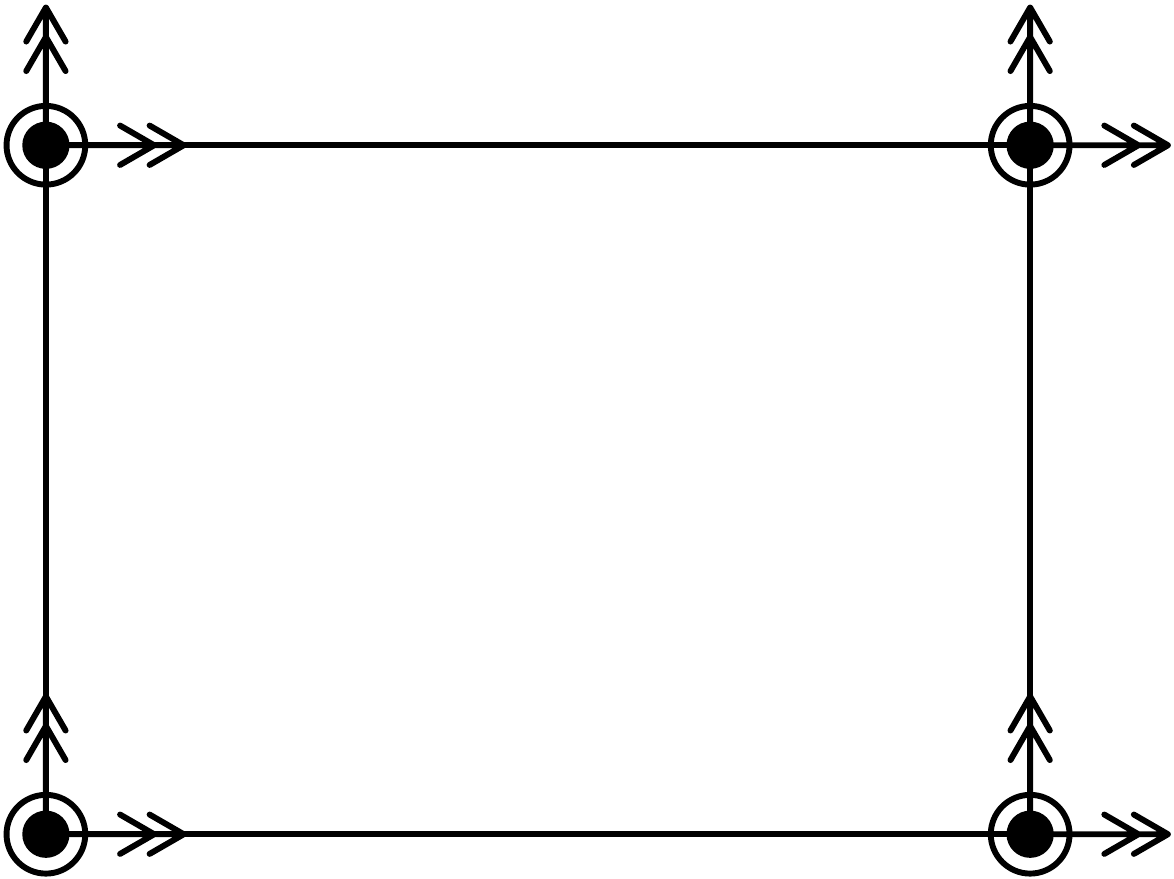} 
\label{fig:rect-element}
} \qquad 
\subfloat[Cubic element.]
{
\includegraphics[width=2.0in]{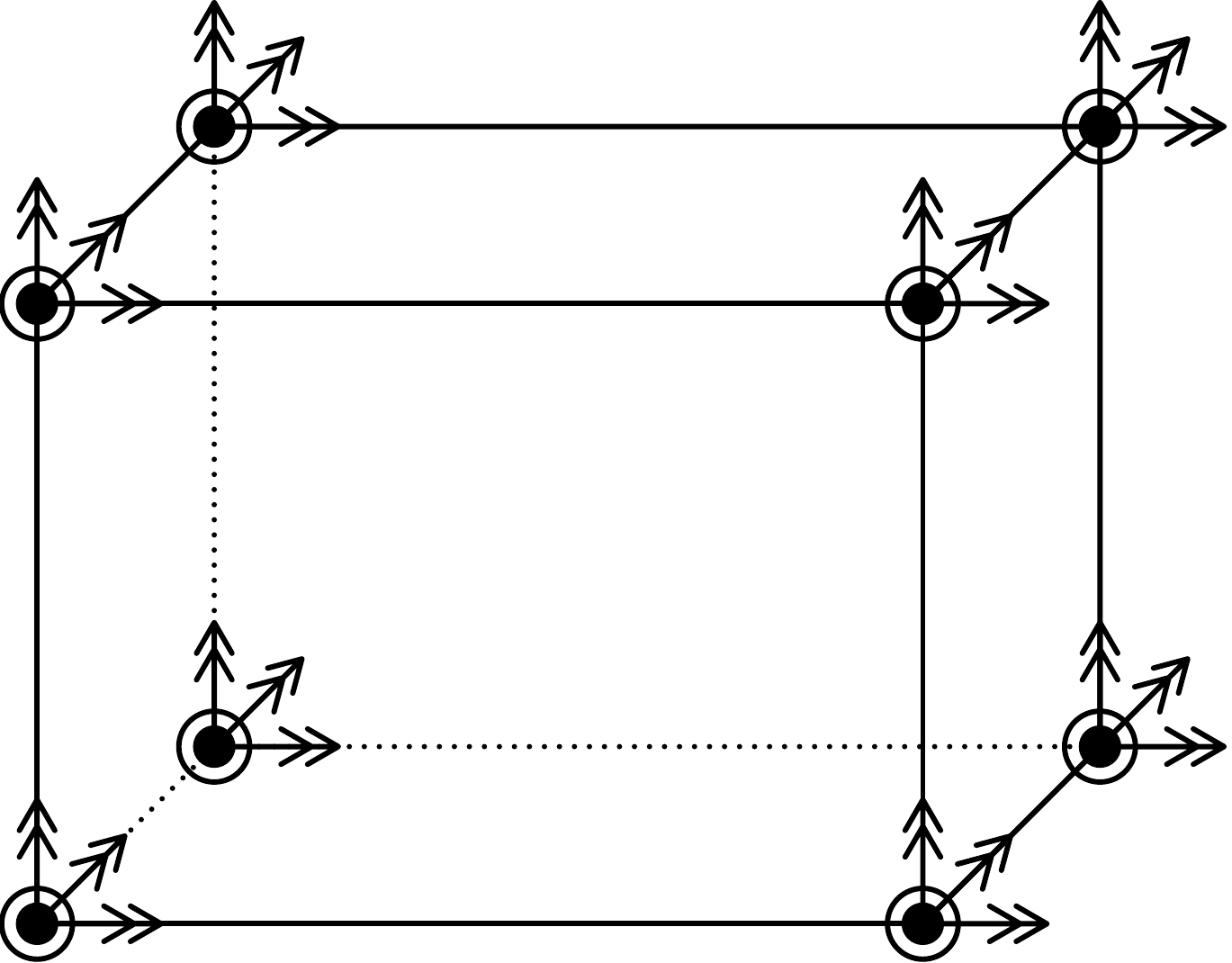} 
\label{fig:cubic-element}
} 
\caption{$H^3$-nonconforming Adini-type element} \label{fig:finite-element}.
\end{figure}

\begin{itemize}
\item $\mathcal{P}_T = \mathcal{P}_{A}(T)$.
\item For $v \in C^2(T)$, the vector $\mathcal{N}_T(v) $ of degree of freedom is
\begin{equation}
\mathcal{N}_T(v) = \left( v(a_1),\nabla v(a_1)^\top, D^2_p v(a_1)^\top, \cdots, v(a_{2^n}),\nabla v(a_{2^n})^\top, D^2_p v(a_{2^n})^\top \right)^\top,
\end{equation}
in which $D^2_p = \left(\dfrac{\partial^2}{\partial x_1^2}, \dfrac{\partial^2}{\partial x_2^2}, \cdots, \dfrac{\partial^2}{\partial x_n^2}\right)^\top$ denotes the vector of all pure second-order differential operators.
\end{itemize}

Instead of writing the explicit formulation of basis functions, below we show the unisolvent property of the Adini-type element using an inductive argument.
\begin{lemma}[Unisolvent property of the Adini-type element] 
For the $n$-dimensional Adini-type element, $\mathcal{N}_T$ is $\mathcal{P}_T$-unisolvent. 
\end{lemma}
\begin{proof}
Since the dimensions of both $\mathcal{P}_A(T)$ and the number of DoFs are $2^n(2n+1)$, it suffices to show that if $v \in \mathcal{P}_A(T)$ vanishes on $\mathcal{N}_T$ then $v = 0$. 

The case in which $n=1$ is standard. Assume that the conclusion is true for $n = k (k \geq 1)$. 

Now let $n = k + 1$. We write $v = v(\xi_1, \xi_2, \cdots, \xi_n)$. On the $k$-dimensional face $F_i^\pm$ on which $\xi_i = \pm 1$, $v$ is a polynomial of $\xi_1, \cdot, \xi_{i-1}, \xi_{i+1}, \cdots, \xi_n$ in $k$-dimensional shape function space $\mathcal{P}_A(F_i^\pm)$. Clearly, $\mathcal{N}_{F_i^\pm}(v)$, which consists of the point-values, gradients, and pure second-order derivatives at vertices of $F_i^\pm$, will vanish from the definition of $\mathcal{N}_T$. Hence, $v|_{F_i^\pm} = 0$ by the inductive assumption. This leads to a factor $\Pi_{i=1}^n(\xi_i^2 - 1)$ of $v$. Consequently, $v = 0$.  $\blacksquare$
\end{proof}

We define the finite element space $V_h$ and $V_{h0}$ as follows: 
$V_{h} = \lbrace v_h\in L^2(\Omega):~ v_h|_{T} \in P_A(T), ~ v_h, \frac{\partial v_h}{\partial x_j}, \frac{\partial^2 v_h}{\partial x_j^2}$ are continuous at all vertices of elements in $\mathcal{T}_h, 1\leq j\leq n \rbrace$, 
and $V_{h0} = \lbrace v_h\in V_h:~  v_h, \frac{\partial v_h}{\partial x_j}, \frac{\partial^2 v_h}{\partial x_j^2}$ vanish at vertices along $\partial \Omega\rbrace$.

From the proof of unisolvent property, we directly see that $V_h \subset H^1(\Omega)$ and $V_{h0} \subset H_0^1(\Omega)$. In fact, when restricting $v\in V_h$ on an $(n-1)$-dimensional face $F$, $v|_F$ is uniquely defined $\mathcal{N}_F$, which yields the continuity of $v$. Further, if $v \in V_{h0}$ and $F\subset \partial \Omega$, then $v|_F = 0$. 

\begin{lemma}[normal-normal strong continuity for Adini] \label{lm:Adini-nn}
Let $V_h$ and $V_{h0}$ be the finite element spaces of the $n$-rectangle Adini-type element. Then, 
\begin{equation} \label{eq:Adini-nn}
\left.\dfrac{\partial^2}{\partial \nu^2} (v|_T)\right|_F = 
\left.\dfrac{\partial^2}{\partial \nu^2} (v|_{T'})\right|_F \quad \forall v \in V_h,
\end{equation}
where $T,T' \in \mathcal{T}_h$ share a common $(n-1)$-dimensional interior face $F$. Moreover, if an $(n-1)$-dimensional face $F$ of $T \in \mathcal{T}_h$ is on $\partial \Omega$, then 
\begin{equation} \label{eq:Adini-nn-bc}
\left.\dfrac{\partial^2}{\partial \nu^2} (v|_T)\right|_F = 0 \quad \forall v \in V_{h0}.
\end{equation}
\end{lemma}
\begin{proof}
We prove the case for $F_{T,i}^{\pm}$ in which $\frac{\partial^2}{\partial \nu^2} = \frac{\partial^2}{\partial x_i^2}$. Recall that $\Pi_h^{\bm{0}}$ is the global $n$-linear interpolation operator to $Q_1$-FEM space, the pure second-order derivatives at vertices belong to the DoFs of the Adini-type element, then $\Pi_h^{\bm{0}} \frac{\partial^2 v}{\partial x_i^2} \in  H^1(\Omega)$. 

Since $v|_T \in Q_1(T)\cdot \operatorname{span}\lbrace 1,x_j^2,x_j^4 ~|~ 1\leq j \leq n \rbrace$, then we have $\frac{\partial^2 (v|_T)}{\partial x_i^2} \in Q_1(T)\cdot \operatorname{span}\lbrace 1,\xi_i^2 \rbrace$ and whence
$$
\left(\dfrac{\partial^2 v}{\partial x_i^2} - \Pi_h^{\bm{0}} \dfrac{\partial^2 v}{\partial x_i^2}\right)\Big|_{F_{T,i}^{\pm}} \in Q_1(F_{T,i}^{\pm}).
$$
Notice that the left-hand side vanishes at all vertices of $F_{T,i}^{\pm}$, which leads to
\begin{equation}
\left. \dfrac{\partial^2 v}{\partial x_i^2}\right|_{F_{T,i}^{\pm}} = \left.\Pi_h^{\bm{0}} \dfrac{\partial^2 v}{\partial x_i^2} \right|_{F_{T,i}^{\pm}} 
\end{equation}
For $v \in V_{h0}$, we have $\Pi_h^{\bm{0}} \frac{\partial^2 v}{\partial x_i^2} \in H_0^1(\Omega)$, which leads to \eqref{eq:Adini-nn-bc}.  $\blacksquare$
\end{proof}

\section{Approximation Property} \label{sc:approximation}

In this section, we consider the approximation property of the Adini-type element and the Morely-type element. The interpolation error analysis of these finite element spaces in any dimension is established by using the projection-averaging technique. In section 3.2 we extend our investigation to some conforming relatives. Following similar ideas, we sketch the proofs of the error estimate and the stability of the conforming interpolation operator.

\subsection{Interpolation error of the $H^3$ nonconforming element}
In this section, we will analyze the approximation property of the finite element spaces $V_h$ and $V_{h0}$. To start with, we have the following result for low-dimensional cases.

\begin{theorem}\label{interpolation-for-nleq3}
Let $\Pi_T$ be the interpolation operator of the $n$-rectangle Morley-type element or the $n$-rectangle Adini-type finite element. If $n\leq 3$ then for any $T \in \mathcal{T}_h$, 
\begin{equation}
|v - \Pi_T v|_{m,T} \lesssim h^{4-m} |v|_{4,T} \quad 0\leq m\leq 4, ~ \forall v \in H^4(T). 
\end{equation}
\end{theorem}

Theorem \ref{interpolation-for-nleq3} can be obtained from the standard interpolation theory (c.f. \cite{ciarlet2002finite}) and the result is already enough for practical cases. However, we are interested in attaining similar results for a more generic case in which $n \geq 2$.

\begin{theorem}[approximation property] \label{thm:approx-of-Vh-Vh0}
Let $V_h$ and $V_{h0}$ be the finite element spaces of the $n$-rectangle Morley-type element or the $n$-rectangle Adini-type element. Then, for any $s\in [0,1]$, 
\begin{align}
&\mathop{\inf}_{v_h \in V_h} \sum_{m=0}^{3} h^m |v-v_h|_{m,h} \lesssim h^{3+s} |v|_{3+s,\Omega} \quad 
\forall v \in H^{3+s}(\Omega)\label{eq:approx-Vh}, \\  
&\mathop{\inf}_{v_h \in V_{h0}} \sum_{m=0}^{3} h^m |v-v_h|_{m,h} \lesssim h^{3+s} |v|_{3+s,\Omega} \quad 
\forall v \in H^{3+s}(\Omega)\cap H_0^3(\Omega).\label{eq:approx-Vh0}
\end{align}
\end{theorem}
\begin{proof}
The proof is based on the well-established projection-averaging technique (c.f. \cite{wang2007some}). For conciseness and completeness, we present the proof of \eqref{eq:approx-Vh0} for the $n$-rectangle Adini-type element. For a function $v\in H^{3+s}(\Omega)\cap H_0^3(\Omega)$, we define $w_h \in L^2(\Omega)$ as the $L^2$-projection of $v$ onto $\mathcal{P}_A(T)$ for each $T \in \mathcal{T}_h$, namely,
\begin{equation*}
w_h|_{T} \in \mathcal{P}_A(T)~\text{and}~ \int_{T} w_hq\,\mathrm{d}x = \int_{T} vq \,\mathrm{d}x, \quad \forall q \in \mathcal{P}_A(T), ~T\in\mathcal{T}_h.
\end{equation*}
Since $\mathcal{P}_3(T) \subset \mathcal{P}_A(T)$, then the standard interpolation theory of $L^2$-projection \cite{brenner2008mathematical} gives the following bound:
\begin{equation}\label{err-L2-projection}
|v-w_h|_{m,T} \lesssim h^{3+s-m}|v|_{3+s,T}, \quad 0\leq m \leq 3, ~T \in \mathcal{T}_h.
\end{equation}

Given a set $B\subset \mathbb{R}^n$, define $\mathcal{T}_h(B)= \lbrace T\in \mathcal{T}_h: ~T \cap B \neq \varnothing \rbrace$ and let $N_h(B)$ be the number of elements in $\mathcal{T}_h(B)$. In what follows, we will use the notation $w_h^T = w_h|_T$ for simplicity. Now we define the interpolation $v_h \in V_{h0}$ by taking the average of the DoFs.  For $a_i$ being an interior vertex of $\Omega$, let
\begin{align}
 v_h(a_i) &:= \dfrac{1}{N_h(a_i)} \sum_{T^{\prime} \in \mathcal{T}_h(a_i)} w_h^{T^{\prime}}(a_i), \quad i = 1,2,\cdots,2^n, \\
 \dfrac{\partial v_h(a_i)}{\partial x_j} &:= \dfrac{1}{N_h(a_i)} \sum_{T^{\prime} \in \mathcal{T}_h(a_i)} \dfrac{\partial w_h^{T^{\prime}}(a_i)}{\partial x_j}, \quad i = 1,2,\cdots,2^n, \quad j = 1,2,\cdots,n,\\
 \dfrac{\partial^2 v_h(a_i)}{\partial x_j^2} &:= \dfrac{1}{N_h(a_i)} \sum_{T^{\prime} \in \mathcal{T}_h(a_i)} \dfrac{\partial^2 w_h^{T^{\prime}}(a_i)}{\partial x_j^2}, \quad i = 1,2,\cdots,2^n, \quad j = 1,2,\cdots,n.
\end{align}
Let $\phi_h:=w_h - v_h$ and obviously $\phi_h^T \in \mathcal{P}_A(T)$ on each $T\in \mathcal{T}_h$. By a standard scaling argument, we find that, for $0 \leq m \leq 3$, 
\begin{equation}\label{scaling-argument}
\vert \phi_h \vert_{m,T}^2 \lesssim h^{n-2m} \left( \sum_{i=1}^{2^n} \vert \phi_h^T(a_i)\vert^2 
+ h^2 \sum_{i=1}^{2^n}\sum_{j=1}^{n} \Big\vert \dfrac{\partial \phi_h^T (a_i)}{\partial x_j} \Big\vert^2 
+ h^4\sum_{i=1}^{2^n}\sum_{j=1}^{n}\Big\vert \dfrac{\partial^2 \phi_h^T (a_i)}{\partial x_j^2} \Big\vert^2\right),
\end{equation}
Next we complete the proof by respectively estimating the terms $\vert \phi_h(a_i)\vert$, $\big\vert \frac{\partial \phi_h (a_i)}{\partial x_j} \big\vert$ and $\big\vert \frac{\partial^2 \phi_h (a_i)}{\partial x_j^2} \big\vert$ in \eqref{scaling-argument}. If $a_i\in T$ is an interior node of $\Omega$, by definition we have 
\begin{equation*}
\phi_h^T(a_i) = \dfrac{1}{N_h(a_i)} \sum_{T^{\prime}\in\mathcal{T}_h(a_i)}\left(w_h^T(a_i) - w_h^{T^{\prime}}(a_i) \right).
\end{equation*}
For any other element $T^{\prime}$ in the patch $\mathcal{T}_h(a_i)$, there exists an integer $J>0$ and $T_1,T_2,\cdots,T_J \in \mathcal{T}_h(a_i)$ such that $T_1=T$, $T_J=T^{\prime}$ and $\tilde{F}_j = T_{j} \cap T_{j+1}$ is a common $(n-1)$-dimensional surface of $T_j$ and $T_{j+1}$, with $a_i \in \tilde{F}_j,~1\leq j\leq J$. A simple computation with the inverse estimate gives
\begin{align*}
\vert w_h^T(a_i) - w_h^{T^{\prime}}(a_i) \vert^2 &= \Big\vert \sum_{j=1}^{J-1} \left(  w_h^{T_{j}}(a_i) - w_h^{T_{j+1}}(a_i) \right) \Big\vert^2 \\ 
&\lesssim h^{1-n}\sum_{j=1}^{J-1}  \| w_h^{T_{j}}- w_h^{T_{j+1}} \|_{0,\tilde{F}_j}^2 \\ 
&\lesssim h^{1-n} \sum_{j=1}^{J-1} \left( \| v-w_h^{T_j} \|_{0,\tilde{F}_j}^2 + \| v-w_h^{T_{j+1}} \|_{0,\tilde{F}_j}^2 \right).
\end{align*}
Taking $m=0,1$ in \eqref{err-L2-projection} and using the local trace theorem, we obtain that
\begin{equation*}
\| v-w_h^{T_j} \|_{0,\tilde{F}_j}^2 \lesssim h^{-1} \Vert v-w_h^{T_j} \Vert_{0,T_j}^2 + h \vert v-w_h^{T_j} \vert_{1,T_j}^2 \lesssim h^{5+2s} \vert v \vert_{3+s,T_j}.
\end{equation*}
Since the values $J$ and $N_h(a_i)$ are uniformly bounded for any interior vertex $a_i$ in $\Omega$, then it is concluded that 
\begin{equation}\label{bound-vertex-value}
\vert \phi_h(a_i) \vert^2  \lesssim h^{6-n+2s} \sum_{T^{\prime} \in \mathcal{T}_h(a_i)} \vert v \vert_{3+s, T^{\prime}}^2.
\end{equation} 

If the vertex $a_i$ of $T$ is on the boundary $\partial \Omega$, then there exist $T^{\prime} \in \mathcal{T}_h(a_i)$ with an $(n-1)$-dimensional face $F \subset \partial \Omega$, such that $a_i \in F$. Therefore, we estimate $\phi_h$ by
\begin{equation*}
\vert \phi_h(a_i) \vert \leq \vert w_h^T(a_i) - w_h^{T^{\prime}}(a_i) \vert + \vert  w_h^{T^{\prime}} (a_i) \vert. 
\end{equation*}
The first term above in the right hand side can be handled with previous technique, and the inverse estimate gives the bound for the second term:
\begin{equation*}
\vert  w_h^{T^{\prime}} (a_i) \vert^2  \lesssim h^{1-n} \|  w_h^{T^{\prime}} \|_{0,F}^2 \eqsim h^{1-n} \|  v - w_h^{T^{\prime}} \|_{0,F}^2 \lesssim h^{6-n+2s} \vert v \vert_{3+s,T^{\prime}}^2.
\end{equation*} 
Therefore, \eqref{bound-vertex-value} also holds for vertices $a_i \in \partial \Omega$. It is noticed that the same analysis can be applied on $\big\vert \frac{\partial \phi_h (a_i)}{\partial x_j} \big\vert$ and $\big\vert \frac{\partial^2 \phi_h (a_i)}{\partial x_j^2} \big\vert$ so that we have the following estimates:
\begin{align}
\Big\vert \dfrac{\partial \phi_h (a_i)}{\partial x_j} \Big\vert^2 &\lesssim h^{4-n+2s} \sum_{T^{\prime}\in \mathcal{T}_h(a_i)} \vert v \vert_{3+s,T^{\prime}}^2, \quad i = 1,2,\cdots,2^n, \quad j = 1,2,\cdots,n,\label{bound-vertex-1stderivative}\\
\Big\vert \dfrac{\partial^2 \phi_h (a_i)}{\partial x_j^2} \Big\vert^2 &\lesssim h^{2-n+2s} \sum_{T^{\prime}\in \mathcal{T}_h(a_i)} \vert v \vert_{3+s,T^{\prime}}^2,\quad i = 1,2,\cdots,2^n, \quad j = 1,2,\cdots,n.\label{bound-vertex-2ndderivative}
\end{align}
Combining \eqref{scaling-argument} with \eqref{bound-vertex-value}-\eqref{bound-vertex-2ndderivative}, and summing over $T\in\mathcal{T}_h$, we have, for $0 \leq m \leq 3$
\begin{equation}\label{bound-phi}
h^{2m} \vert \phi_h \vert_{m, h}^2 \lesssim h^{6+2s} \vert v \vert_{3+s, \Omega}^2.
\end{equation}
The result \eqref{eq:approx-Vh0} follows from \eqref{bound-phi}, \eqref{err-L2-projection}, and the triangle inequality. $\blacksquare$
\end{proof}

\subsection{Conforming relatives}
Introduced by Brenner in \cite{brenner1996twolevel}, the conforming relative of a nonconforming finite element is verified to be capable of reducing the regularity requirements in the convergence analysis (e.g.~\cite{wu2019nonconforming}). Let us now consider a family of $H^3$ conforming elements on $n$ dimensional rectangle meshes. For any integers $k\geq 0$, define the set of degree of freedom of an $H^{k+1}$ $n$-rectangle finite element as follows.
\begin{equation}\label{eq:high-order-dof}
\mathcal{N}_T^k(v) = \Big\lbrace \dfrac{\partial^{\alpha} v}{\partial x^\alpha} (a_i)~:~ 0\leq \alpha_j \leq k, ~j = 1,2,\cdots,n, ~i = 1,2,\cdots, 2^n \Big\rbrace,
\end{equation}
where $a_i, 1\leq i \leq 2^n$ are vertices of the $n$-rectangle $T$. The corresponding shape function space of $\mathcal{N}_T^k$ on $T \in \mathcal{T}_h$ is therefore $Q_{2k+1}(T)$. Next we let $V_{h}^k$, $V_{h0}^k$ be the global finite element space on the domain $\Omega$. By regarding $\mathcal{N}_T^k$ as a tensor product of $n$ set of degree of freedoms of $(2k+1)$-th order Hermitian interpolation in one dimension,  it can be shown that $V_{h}^k \subset H^{k+1}(\Omega)$ through mathematical induction on the dimensionality $n$.

In the following we still borrow the notations of the projection-averaging strategy described in Theorem~\ref{thm:approx-of-Vh-Vh0} to construct the interpolation operators of functions with less smoothness. Based on the existence of the conforming relative with arbitrary regularities, we have following conclusion.
\begin{lemma}[Approximation property of $H^3$ conforming relative]\label{lemma: conforming-relative-H3}
There exists an $H^3$-conforming $n$-rectangle finite element space $V_h^c \subset H_0^3(\Omega)$ and an interpolation operator $\Pi_h^{c}: V_h \rightarrow V_h^c$ such that
\begin{equation}\label{err-conforming-relative-H3}
\sum_{m=0}^{3} h^{m-3} \vert v_h - \Pi_h^c v_h \vert_{m,h} \lesssim \vert v_h \vert_{3,h}, ~\forall v_h \in  V_h.
\end{equation}
\end{lemma}

\textit{Sketch of Proof.} Note that for any $v_h \in V_h$, it holds that $v_h|_T \in Q_5(T)$.
Taking $k=2$ in~\eqref{eq:high-order-dof} and $V_h^c = V_{h0}^2$, the interpolation operator $\Pi_h^c$ is then defined as follows. For $a_i$ being an interior vertex node of $\mathcal{T}_h$ and $d_T \in \mathcal{N}_T^2$ being any one of the degree of freedoms, let
\begin{equation}
d_T(\Pi_h^c v_h)(a_i) = \dfrac{1}{N_h(a_i)} \sum_{T^{\prime} \in \mathcal{T}_h(a_i)} d_{T^\prime} (v_h^{T^\prime}) (a_i).
\end{equation}
Here, $d_{T^\prime}$ should be of the same type as $d_T$ and $T^\prime$ shares the same vertex node $a_i$ with $T$.  For $a_i \in \partial \Omega$ being a boundary vertex, we then define $d_T(\Pi_h^c v_h)(a_i) = 0$.  The rest of the estimation is highly similar to the proof of Theorem~\ref{thm:approx-of-Vh-Vh0} and we ommit here for brevity. $\blacksquare$

\begin{lemma}[Approximation property of $H^4$ conforming relative]\label{lemma: conforming-relative-H4}
Let $s\in [0,1]$ and $u\in H^{3+s}(\Omega)\cap H^{3}_0(\Omega)$, there exists an $n$-rectangle finite element space $\tilde{V}_h \subset H^4(\Omega) \cap H_0^3(\Omega)$ and an interpolation operator $\tilde{\Pi}_h: H^{3+s}(\Omega)\cap H^{3}_0(\Omega)\rightarrow \tilde{V}_h $ such that
\begin{equation}\label{err-conforming-relative-H4}
\sum_{m=0}^{3} h^{m-3-s}\vert u - \tilde{\Pi}_h u \vert_{m,h} + \vert \tilde{\Pi}_h u \vert_{3+s,\Omega} \lesssim \vert u \vert_{3+s,\Omega}
\end{equation}
\end{lemma}

\textit{Sketch of Proof.} 
Firstly we consider taking $k=3$ in \eqref{eq:high-order-dof} to obtain a finite element space $V_{h}^3 \subset H^4(\Omega)$ and the set of DoFs $\mathcal{N}_T^3$. In order to maintain the boundary conditions of $H_0^3(\Omega)$, some necessary corrections should be made such that $\tilde{V}_h \subset V_{h}^3 \cap H_0^3(\Omega)$. For $u\in H^{3+s}(\Omega) \cap H_0^3(\Omega)$, define $w_h \in L^2(\Omega)$ such that
\begin{equation}
w_h|_{T}  := w_h^{T} \in Q_7(T) \text{  and } \int_{T} w_h q \,\mathrm{d}x = \int_{T} u q\,\mathrm{d}x, ~ \forall q \in Q_7(T), ~ T\in \mathcal{T}_h.
\end{equation}
Then the interpolation $\tilde{\Pi}_h u$ is given by using $\mathcal{N}_T^3$ and evaluated as
\begin{equation}\label{eq:correction-of-H4-dofs}
d_T(\tilde{\Pi}_{h} u)(a_i) = 
\left\{
\begin{aligned}
&0 ,& ~&\text{if } d_T(v)(a_i) = 0,~ \forall v\in H_0^3(\Omega) \cap C^{\infty}(\Omega),\\
&\dfrac{1}{N_h(a_i)} \sum_{T^{\prime} \in \mathcal{T}_h(a_i)} d_{T^\prime} (w_h^{T^\prime}) (a_i),& ~&\text{otherwise.}
\end{aligned}
\right.
\end{equation}
We note here the first condition of \eqref{eq:correction-of-H4-dofs} only guarantees part of the DoFs to be zero on boundary vertices. Again, we refer to the proof of Theorem~\ref{thm:approx-of-Vh-Vh0} for the rest of the estimation, following which we also have
\begin{align*}
&\vert u - \tilde{\Pi}_h u \vert_{3,\Omega} \lesssim \vert u \vert_{3,\Omega},  \text{ for } u \in H_0^3(\Omega), \\
& \vert u - \tilde{\Pi}_h u \vert_{4,\Omega} \lesssim \vert u \vert_{4,\Omega},  \text{ for } u \in H^4(\Omega) \cap H_0^3(\Omega).
\end{align*}
This gives the stability result $\vert u - \tilde{\Pi}_h u \vert_{3+s,\Omega} \lesssim \vert u \vert_{3+s,\Omega} $ for any $s\in[0,1]$ by applying the interpolation theory of the Sobolev spaces. 
$\blacksquare$

\section{Estimate of tangential-normal terms by $n$-rectangle interpolation} \label{sc:tn}
From the convergence framework of nonconforming methods \cite{wang2001on}, the weak continuities are crucial in the analysis. In terms of the $H^3$ problems, one needs to take care of all the second-order derivatives, which consist of the tangential-tangential, normal-normal, and tangential-normal components. For the Morley-type element, the tangential-tangential and normal-normal continuities are weak, see Lemmas \ref{lm:Morley-tt} and \ref{lm:Morley-nn}, respectively. Thanks to the $C^0$-continuity of Adini-type finite element space and Lemma \ref{lm:Adini-nn}, the tangential-tangential and normal-normal components are strongly continuous. 

The rest of the second-order terms, i.e. the tangential-normal terms, can not be tackled via the DoFs. As a special property of the $n$-rectangle element, the interpolation is a crucial tool in the convergence analysis. 

\subsection{Some properties by local interpolation}

We derive several local interpolation properties. Let us denote the $(n-2)$-dimensional sub-rectangles of $T$ as:
\begin{equation} \label{eq:edge}
\ell_{T,i,j}^{\pm,\pm} = \lbrace x\in \bar{T} ~\vert ~ \xi_i = \pm 1, \xi_j = \pm 1 \rbrace 
\quad \text{for } j\neq i.
\end{equation}

\begin{lemma}[Properties of Morley-type element by local interpolation] \label{lm:Morley-tn-local}
Let $v \in \mathcal{P}_M(T) $. For $j \neq i$, it holds that 
\begin{equation} \label{eq:Morley-tn-local}
\int_{F_j^\pm} \dfrac{\partial }{\partial x_i} \left( \dfrac{\partial (\Pi_T^{\bm{1}} v) }{\partial x_i} - \Pi_T^{\bm{0}} \dfrac{\partial (\Pi_T^{\bm{1}} v) }{\partial x_i} \right) \,\mathrm{d} S = 0,
\end{equation}
where $\Pi_T^{\bm{0}}$ and $\Pi_T^{\bm{1}}$ are the local interpolations of $Q_1$ and Adini elements, respectively (see \eqref{eq:global-Q1-Pi}).
\end{lemma}
\begin{proof} We have $\Pi_T^{\bm{1}} v \in Q_1(T) \cdot \operatorname{span}\{1, x_k^2~|~ 1 \leq k \leq n\}$, and hence 
$$ 
\dfrac{\partial (\Pi_T^{\bm{1}} v) }{\partial x_i} \in Q_1(T) + Q_1^{\hat{i}}(T) \cdot \operatorname{span}\{ \xi_i^2 - 1\} := Q_1(T) + G_i(T).
$$ 
Next, we observe that both $\frac{\partial (\Pi_T^{\bm{1}} v) }{\partial x_i} - \Pi_T^{\bm{0}} \frac{\partial (\Pi_T^{\bm{1}} v) }{\partial x_i}$ and $G_i(T)$ vanish at the vertices of $T$, whence
$$ 
\dfrac{\partial (\Pi_T^{\bm{1}} v) }{\partial x_i} - \Pi_T^{\bm{0}} \dfrac{\partial (\Pi_T^{\bm{1}} v) }{\partial x_i} \in G_i(T).
$$ 
Notice that $G_i(T)$ vanishes on $(n-2)$-dimensional sub-rectangles
  $\ell_{T,i,j}^{\pm,\pm}$ due to the factor $(\xi_i^2 - 1)$. Then,
  the desired result \eqref{eq:Morley-tn-local} can be obtained by
  integrating along the $x_i$ direction. $\blacksquare$
\end{proof}

\begin{lemma}[Properties of Adini-type element by local interpolation] \label{lm:Adini-tn-local}
Let $v \in \mathcal{P}_A(T) $. For $j \neq i$, it holds that 
\begin{equation} \label{eq:Adini-tn-local}
\int_{F_j^\pm} \dfrac{\partial }{\partial x_i} \left( \dfrac{\partial v }{\partial x_i} - \Pi_T^{\bm{e}_i} \dfrac{\partial v}{\partial x_i} \right)\,\mathrm{d}S = 0,
\end{equation}
where $\Pi_T^{\bm{e}_i}$ are the local interpolation of the partial Adini element (see \eqref{eq:global-Q1-Pi}). 
\end{lemma}
\begin{proof}
For any $v \in \mathcal{P}_A(T) = Q_1(T) \cdot \operatorname{span}\{1, x_k^2, x_k^4 ~|~ 1 \leq k \leq n\}$, we have 
$$
\begin{aligned}
\frac{\partial v}{\partial x_i} & \in  Q_1^{\hat{i}}(T) \cdot \operatorname{span} \{1, x_k^2, x_k^4~|~ 1 \leq k \leq n\} + Q_1(T) \cdot \operatorname{span}\{x_i, x_i^3\}  \\
& = Q_1(T) \cdot \mathrm{span}\{1, \xi_i^2\} + Q_1^{\hat{i}}(T) \cdot \mathrm{span}\{(\xi_k^2 -1), (\xi_t^2 -1)^2 ~|~ k\neq i, 1\leq t \leq n\} \\
& := Q_1(T) \cdot \mathrm{span}\{1, \xi_i^2\}  + W_i(T).
\end{aligned}
$$
Next, we see that for any $w \in W_i(T)$, $w$ and $\frac{\partial w}{\partial x_i}$ vanish at the vertices of $T$, which exactly correspond to the DoFs of $n$-rectangle partial Adini element. Therefore, 
$$
\dfrac{\partial v }{\partial x_i} - \Pi_T^{\bm{e}_i} \dfrac{\partial v}{\partial x_i}  \in W_i(T).
$$
Now, let $\alpha_k, \beta_t \in \mathbb{R}$ and $q_k, r_t \in Q_1^{\hat{i}}(T)$ such that
$$
\frac{\partial v }{\partial x_i} - \Pi_T^{\bm{e}_i} \frac{\partial v}{\partial x_i} = \sum_{k \neq i} \alpha_k q_k (\xi_k^2-1) + \sum_{t = 1}^n \beta_t r_t (\xi_k^2 - 1)^2.
$$
Then, we obtain 
$$ 
\int_{F_j^\pm} \dfrac{\partial }{\partial x_i} \left( \dfrac{\partial v }{\partial x_i} - \Pi_T^{\bm{e}_i} \dfrac{\partial v}{\partial x_i} \right)\,\mathrm{d}S 
= \int_{F_j^\pm}  \dfrac{\partial }{\partial x_i} \left(\beta_i r_i (\xi_i^2 - 1)^2 \right) \,\mathrm{d}S = 0.
$$ 
This completes the proof. $\blacksquare$
\end{proof}

\subsection{Estimate of tangential-normal terms: Exchange of sub-rectangles}

We use a new technique called {\it exchange of sub-rectangles} to estimate the tangential-normal terms. 

\begin{lemma}[Estimate of tangential-norm terms]\label{lm:err-tn}
Let $\phi \in H^1(\Omega)$ be a piecewise polynomial defined on $\mathcal{T}_h$, $V_{h0}$ be the finite element space of the $n$-rectangle Morley-type element or the $n$-rectangle Adini-type element. For $j \neq i$, it holds that 
\begin{equation}\label{eq:err-tn}
\Big\vert \sum_{T\in\mathcal{T}_h}\int_{\partial T} \phi \dfrac{\partial^2 v_h}{\partial x_i \partial x_j} \nu_i \,\mathrm{d}S \Big\vert \leq C h \vert \phi \vert_{1,\Omega} |v_h |_{3,h}.
\end{equation}
\end{lemma}
\begin{proof} For the sake of simplicity of the exposition, we first show \eqref{eq:err-tn} for the Adini-type element, then sketch the proof for the Morly-type element. 

\underline{Part I: proof for Adini-type element.} It is readily seen that $\nu_i |_{F_{T,i}^{\pm}} = \pm 1$ and vanishes on other $(n-1)$-dimensional faces of $T$. Then, using integration by parts on $F_{T,i}^\pm$, we have
\begin{equation}\label{eq:Adini-tn1}
\begin{aligned}
\sum_{T\in\mathcal{T}_h}\int_{\partial T} \phi \dfrac{\partial^2 v_h}{\partial x_i \partial x_j} \nu_i \,\mathrm{d}S &= \sum_{T\in\mathcal{T}_h}  \int_{F_{T,i}^{+}+F_{T,i}^{-}}\phi \dfrac{\partial^2 v_h}{\partial x_i \partial x_j}\nu_{i} \,\mathrm{d}S
= \sum_{T\in\mathcal{T}_h}  \int_{F_{T,i}^{+} - F_{T,i}^{-}}\phi \dfrac{\partial^2 v_h}{\partial x_i \partial x_j}\,\mathrm{d}S
 \\
& = \sum_{T\in\mathcal{T}_h} \int_{\partial F_{T,i}^{+} - \partial F_{T,i}^{-}} \phi \dfrac{\partial v_h}{\partial x_i} \nu_j \,\mathrm{d}\ell  - \sum_{T\in\mathcal{T}_h}
\int_{F_{T,i}^{+}-F_{T,i}^{-}} \dfrac{\partial \phi}{\partial x_j} \dfrac{\partial v_h}{\partial x_i} \,\mathrm{d}S := I_1 + I_2.
\end{aligned}
\end{equation}
Here, with a little bit abuse of notation, $\nu_j$ represents the $j$-th component of the unit outer vector which is normal to $\partial F_{T,i}^\pm$ and parallel to $F_{T,i}$.

\underline{Analysis of $I_2$}. Recall that $\Pi_h^{c}$ is the interpolation operator of the conforming relative defined in Lemma~\ref{lemma: conforming-relative-H3}. Notice that the inverse inequality can be applied on $\phi$ and that $\frac{\partial \phi}{\partial x_j}$ is actually continuous across the surfaces $F_{T,i}^{\pm}$ due to the $C^0$-continuity of $\phi$. Therefore, using the trace theorem, the estimate of $\Pi_h^c$ and the interpolation error \eqref{err-conforming-relative-H3} gives the estimate
\begin{equation} \label{eq:exchange-I2}
\begin{aligned}
\vert I_2 \vert &= \Big \vert \sum_{T\in\mathcal{T}_h}\int_{F_{T,i}^{+}-F_{T,i}^{-}} \dfrac{\partial \phi}{\partial x_j} \dfrac{\partial}{\partial x_i} \left(v_h-\Pi_h^c v_h\right) \,\mathrm{d}S \big \vert  \\ 
&\lesssim \sum_{T\in\mathcal{T}_h} \vert \phi \vert_{1,\partial T} \Big\Vert \dfrac{\partial}{\partial x_i} \left(v_h-\Pi_h^c v_h\right) \Big\Vert_{0,\partial T} 
\lesssim \sum_{T\in\mathcal{T}_h} h_T \vert \phi \vert_{1,T} \vert v_h \vert_{3,T} 
\lesssim h \vert \phi \vert_{1,\Omega} \vert v_h\vert_{3,h}.
\end{aligned}
\end{equation}

\underline{Analysis of $I_1$}.  Note that $\Pi_h^{\bm{e}_i} \frac{\partial v_h}{\partial x_i} \in H_0^1(\Omega)$. Hence, the following identity holds:
\begin{equation*}
I_1 = \sum_{T\in\mathcal{T}_h} \int_{\partial F_{T,i}^{+} - \partial F_{T,i}^{-}} \phi \left( \dfrac{\partial v_h}{\partial x_i}  - \Pi_{h}^{\bm{e}_i} \dfrac{\partial v_h}{\partial x_i}\right) \nu_j  \,\mathrm{d}\ell.
\end{equation*}

\begin{figure}[H]
\centering
\includegraphics[width=5.0in]{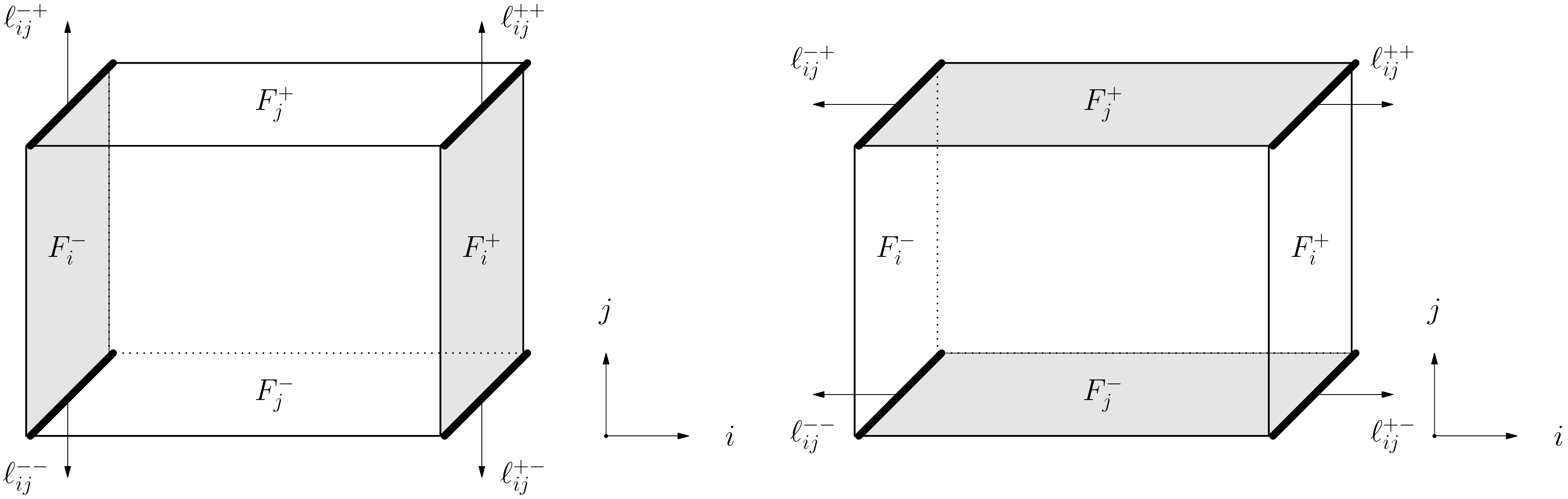} 
\caption{Exchange of sub-rectangles} \label{fig:exchange}.
\end{figure}

Rearranging the integrals over the edges and using the integration by parts, we find
\begin{align*}
I_1 &= \sum_{T\in\mathcal{T}_h} \left(\int_{\ell_{T,i,j}^{+,+} - \ell_{T,i,j}^{+,-}}\phi \left( \dfrac{\partial v_h}{\partial x_i}  - \Pi_{h}^{\bm{e}_i} \dfrac{\partial v_h}{\partial x_i}\right)  \,\mathrm{d}\ell - \int_{\ell_{T,i,j}^{-,+} - \ell_{T,i,j}^{-,-}} \phi \left( \dfrac{\partial v_h}{\partial x_i}  - \Pi_{h}^{\bm{e}_i} \dfrac{\partial v_h}{\partial x_i}\right)  \,\mathrm{d}\ell \right) \\ 
& =  \sum_{T\in\mathcal{T}_h} \left(\int_{\ell_{T,i,j}^{+,+} - \ell_{T,i,j}^{-,+}}\phi \left( \dfrac{\partial v_h}{\partial x_i}  - \Pi_{h}^{\bm{e}_i} \dfrac{\partial v_h}{\partial x_i}\right)  \,\mathrm{d}\ell - \int_{\ell_{T,i,j}^{+,-} - \ell_{T,i,j}^{-,-}} \phi \left( \dfrac{\partial v_h}{\partial x_i}  - \Pi_{h}^{\bm{e}_i} \dfrac{\partial v_h}{\partial x_i}\right)  \,\mathrm{d}\ell \right) \\
& = \sum_{T\in\mathcal{T}_h} \int_{\partial F_{T,j}^{+} - \partial F_{T,j}^{-}}\phi \left( \dfrac{\partial v_h}{\partial x_i}  - \Pi_{h}^{\bm{e}_i} \dfrac{\partial v_h}{\partial x_i}\right)\nu_i  \,\mathrm{d}\ell \\ 
& = \underbrace{\sum_{T\in\mathcal{T}_h} \int_{F_{T,j}^{+} - F_{T,j}^{-}} \dfrac{\partial \phi}{\partial x_i}\left( \dfrac{\partial v_h}{\partial x_i}  - \Pi_{h}^{\bm{e}_i} \dfrac{\partial v_h}{\partial x_i}\right)  \,\mathrm{d}S}_{I_{11}}
+ \underbrace{\sum_{T\in\mathcal{T}_h} \int_{F_{T,j}^{+} - F_{T,j}^{-}} \phi  \dfrac{\partial}{\partial x_i}\left( \dfrac{\partial v_h}{\partial x_i}  - \Pi_{h}^{\bm{e}_i} \dfrac{\partial v_h}{\partial x_i}\right)  \,\mathrm{d}S}_{I_{12}}.
\end{align*}
Here, the second equality applies a new trick called {\it exchange of sub-rectangles} (see Figure \ref{fig:exchange}). Again, the $C^0$-continuity of $\frac{\partial \phi}{\partial x_i} $ across the faces $F_{T,j}^{\pm}$ provides
\begin{equation} \label{eq:exchange-I11}
|I_{11}| = \left|\sum_{T\in\mathcal{T}_h} \int_{F_{T,j}^{+} - F_{T,j}^{-}}\dfrac{\partial \phi}{\partial x_i}\dfrac{\partial }{\partial x_i} \left( v_h -\Pi_h^c v_h\right)  \,\mathrm{d}S \right| \lesssim h\vert \phi \vert_{1,\Omega} \vert v_h \vert_{3,h}.
\end{equation}

Now let $P_F^0: L^2(F) \to \mathcal{P}_0(F)$ be the orthogonal projection. Thanks to Lemma \ref{lm:Adini-tn-local} (Properties of Adini-type element by local interpolation), we obtain 
\begin{equation} \label{eq:exchange-I12}
\begin{aligned}
|I_{12}| &= \left|\sum_{T\in\mathcal{T}_h} \int_{F_{T,j}^{+} - F_{T,j}^{-}} \phi  \dfrac{\partial}{\partial x_i}\left( \dfrac{\partial v_h}{\partial x_i}  - \Pi_{h}^{\bm{e}_i} \dfrac{\partial v_h}{\partial x_i}\right)  \,\mathrm{d}S\right| \\
&=  \left|\sum_{T\in\mathcal{T}_h} \int_{F_{T,j}^{+} - F_{T,j}^{-}} (\phi - P^0_{F} \phi)  \dfrac{\partial}{\partial x_i}\left( \dfrac{\partial v_h}{\partial x_i}  - \Pi_{h}^{\bm{e}_i} \dfrac{\partial v_h}{\partial x_i}\right)  \,\mathrm{d}S\right| \\
&=  \left|\sum_{T\in\mathcal{T}_h} \int_{F_{T,j}^{+} - F_{T,j}^{-}} (\phi - P^0_{F} \phi)  \dfrac{\partial^2}{\partial x_i^2}(v_h  - \Pi_{h}^c v_h) \,\mathrm{d}S\right| 
\lesssim h\vert \phi \vert_{1,\Omega} \vert v_h \vert_{3,h}.
\end{aligned}
\end{equation} 
Combining \eqref{eq:exchange-I2}, \eqref{eq:exchange-I11} and \eqref{eq:exchange-I12}, we finish the proof for the Adini-type element. 

\underline{Part II: Sketch of the proof for Morley-type element.} We recall the special property of Morley-type element \eqref{eq:Morley-Adini}, and consider the fact that the basis functions $r_k^\pm$ defined in \eqref{eq:Morely-basis} depend only on the single variable $x_k$. Then, 
$$
\begin{aligned}
\sum_{T\in\mathcal{T}_h}\int_{\partial T} \phi \dfrac{\partial^2 v_h}{\partial x_i \partial x_j} \nu_i \,\mathrm{d}S 
&= \sum_{T\in\mathcal{T}_h}\int_{\partial T} \phi \dfrac{\partial^2 (\Pi_h^{\bm{1}} v_h)}{\partial x_i \partial x_j} \nu_i \,\mathrm{d}S  \\
& = \sum_{T\in\mathcal{T}_h} \int_{\partial F_{T,i}^{+} - \partial F_{T,i}^{-}} \phi \dfrac{\partial (\Pi_h^{\bm{1}}v_h)}{\partial x_i} \nu_j \,\mathrm{d}\ell  - \sum_{T\in\mathcal{T}_h}
\int_{F_{T,i}^{+}-F_{T,i}^{-}} \dfrac{\partial \phi}{\partial x_j} \dfrac{\partial (\Pi_h^{\bm{1}}v_h)}{\partial x_i} \,\mathrm{d}S := \tilde{I}_1 + \tilde{I}_2.
\end{aligned}
$$
The estimate of $\tilde{I}_2$ is then similar to \eqref{eq:exchange-I2}, by noticing that $\Pi_T^{\bm{1}}$ (local projection of Adini-type element) preserves $\mathcal{P}_3(T)$, namely,
$$ 
\begin{aligned}
|\tilde{I}_2| &\leq \Big \vert \sum_{T\in\mathcal{T}_h}\int_{F_{T,i}^{+}-F_{T,i}^{-}} \dfrac{\partial \phi}{\partial x_j} \dfrac{\partial}{\partial x_i} \left(v_h-\Pi_h^c v_h\right) \,\mathrm{d}S \big \vert + \Big \vert \sum_{T\in\mathcal{T}_h}\int_{F_{T,i}^{+}-F_{T,i}^{-}} \dfrac{\partial \phi}{\partial x_j} \dfrac{\partial}{\partial x_i} \left(v_h-\Pi_h^{\bm{1}} v_h\right) \,\mathrm{d}S \big \vert \\
& \lesssim h |\phi|_{1,\Omega} |v_h|_{3,h}.
\end{aligned}
$$ 
For $\tilde{I}_1$, we insert a global $C^0$ $Q_1$-projection of $\frac{\partial (\Pi_h^{\bm{1}} v_h) }{\partial x_i}$ to obtain that
\begin{equation*}
\tilde{I}_1 = \sum_{T\in\mathcal{T}_h} \int_{\partial F_{T,i}^{+} - \partial F_{T,i}^{-}} \phi \left( \dfrac{\partial (\Pi_h^{\bm{1}}v_h)}{\partial x_i}  - \Pi_{h}^{\bm{0}} \dfrac{\partial (\Pi_h^{\bm{1}}v_h)}{\partial x_i}\right) \nu_j  \,\mathrm{d}\ell.
\end{equation*}
Then the estimate follows from the similar trick (exchange of sub-rectangles) by involving Lemma \ref{lm:Morley-tn-local} (local projection of Morley-type element). $\blacksquare$
\end{proof}

\section{Convergence Analysis and Error Estimate} \label{sc:convergence}
In this section, we will give the convergence analysis of the elements and the error estimate for solving the sixth-order partial differential equations. Given $f \in L^2(\Omega)$, we consider the following tri-harmonic equation:
\begin{equation}\label{eq:tri-harmonic}
\left\{
    \begin{aligned}
        (-\Delta)^3 u  = f &\quad\text{in } \Omega,\\
        u = \dfrac{\partial u}{\partial \nu} = \dfrac{\partial^2 u}{\partial \nu^2}=0 &\quad \text{on } \partial\Omega,
     \end{aligned}
\right.
\end{equation}
where $\Delta$ is the standard Laplacian operator.
Define the bilinear form
\begin{equation} \label{eq:bilinear-a}
a(w,v) = \int_{\Omega} \nabla^3 w : \nabla^3 v\,\mathrm{d}x 
= \int_{\Omega} \sum_{i,j,k=1}^{n} \dfrac{\partial^3 w}{\partial x_i \partial x_j \partial x_k} \dfrac{\partial^3 v}{\partial x_i \partial x_j \partial x_k}\,\mathrm{d}x \quad \forall w,v \in H^3(\Omega).
\end{equation} 
Then, the weak form for the equation \eqref{eq:tri-harmonic} is to find $u\in H_0^3(\Omega)$ such that
\begin{equation} \label{eq:weak-a}
a(u,v) = (f,v) \quad \forall v \in H_0^3(\Omega).
\end{equation}

Since the finite element spaces $V_h$ are $H^3$-nonconforming, we define a discrete bilinear form for $\forall w, v \in L^2(\Omega)$ with $w|_{T}, v|_{T} \in H^3(T), \forall T \in \mathcal{T}_h$, 
\begin{equation} \label{eq:bilinear-ah}
a_h(w,v)=\sum_{T\in\mathcal{T}_h}\int_{T}\sum_{i,j,k=1}^{n} \dfrac{\partial^3 w}{\partial x_i \partial x_j \partial x_k} \dfrac{\partial^3 v}{\partial x_i \partial x_j \partial x_k}\,\mathrm{d}x.
\end{equation}
Corresponding to the $n$-rectangle Morley-type element or the $n$-rectangle Adini-type element, the finite element method for \eqref{eq:tri-harmonic} is to find $u_h \in  V_{h0}$ such that
\begin{equation}\label{eq:weak-ah}
a_h(u_h, v_h) = (f, v_h) \quad \forall v_h \in V_{h0}.
\end{equation}


We are in the position to estimate the consistency error:
\begin{theorem}[Consistency error] \label{tm:consistency}
Let $V_{h0}$ be the finite element space of the $n$-rectangle Morley-type element or the $n$-rectangle Adini-type element. If $u \in H^{3+s}(\Omega) \cap H_0^3(\Omega)$ for $s \in [0,1]$ and $f \in L^2(\Omega)$, then we have
\begin{equation}\label{err-inconsistency}
\vert a_h(u,v_h) - (f,v_h) \vert \lesssim (h^s \vert u \vert_{3+s, \Omega} + h^3\|f\|_{0,\Omega}) |v_h|_{3,h} \quad \forall v_h \in V_{h0}.
\end{equation}
\end{theorem}

\begin{proof}
Following the notation in Lemma~\ref{lemma: conforming-relative-H4} (approximation property of $H^4$ conforming relative), we take $w_h := \tilde{\Pi}_h u \in \tilde{V}_h$ as the conforming approximation of $u$. Then, the consistency error can be written as
$$
a_h(u,v_h) - (f,v_h) = a_h(u-w_h, v_h - \Pi_h^c v_h) + a_h(w_h, v_h - \Pi_h^c v_h) - (f, v_h - \Pi_h^c v_h).
$$
Thanks to Lemma~\ref{lemma: conforming-relative-H3}(approximation property of $H^3$ conforming relative), the first and the third term can be estimated by
\begin{align}
&\vert a_h(u-w_h, v_h - \Pi_h^c v_h) \vert \lesssim \vert u-w_h \vert_{3,h}\vert  v_h - \Pi_h^c v_h \vert_{3,h} \lesssim \vert u-w_h \vert_{3,h} \vert v_h \vert_{3,h} \label{eq:consistency-1}\\ 
& \vert (f, v_h - \Pi_h^c v_h)  \vert \lesssim \|f\|_{0,\Omega} \|v_h - \Pi_h^c v_h\|_{0,\Omega} \lesssim h^3 \|f\|_{0,\Omega} \vert v_h \vert_{3,h}.\label{eq:consistency-2}
\end{align}
For the middle term of the consistency error, we have 
\begin{align*}
a_h(w_h, v_h - \Pi_h^c v_h)  &= \sum_{T\in\mathcal{T}_h} \int_{T} \nabla^3 w_h : \nabla^3 (v_h - \Pi_h^c v_h)\,\mathrm{d}x \\
& =  \underbrace{\sum_{T\in\mathcal{T}_h} \int_{\partial T} \dfrac{\partial}{\partial \nu}(\nabla^2 w_h) : \nabla^2 (v_h - \Pi_h^c v_h)\, \mathrm{d}S}_{:= E_1} 
\underbrace{-\sum_{T\in\mathcal{T}_h} \int_{T} \nabla^2(\Delta w_h) : \nabla^2 (v_h - \Pi_h^c v_h)\,\mathrm{d}x.}_{:=E_2}
\end{align*}
Using the $C^3$-continuity of $w_h$ in Lemma \ref{lemma: conforming-relative-H4} (approximation property of $H^4$ conforming relative) and $C^2$-continuity of $\Pi_h^c v_h$ in Lemma \ref{lemma: conforming-relative-H3} (approximation property of $H^3$ conforming relative), we find
\begin{equation} \label{eq:E1-1}
\begin{aligned}
E_1 & = \sum_{T \in \mathcal{T}_h}\int_{\partial T} \dfrac{\partial}{\partial \nu}(\nabla^2 w_h) : \nabla^2 v_h\,\mathrm{d}S \\ 
& = \sum_{T \in \mathcal{T}_h}\int_{\partial T} \dfrac{\partial^3 w_h}{\partial \nu^3}\dfrac{\partial^2 v_h}{\partial \nu^2}\,\mathrm{d}S 
+2 \sum_{T \in \mathcal{T}_h} \sum_{j=1}^{n-1}\int_{\partial T} \dfrac{\partial^3 w_h}{\partial \nu^2 \partial \tau_j}\dfrac{\partial^2 v_h}{\partial \nu \partial \tau_j} \,\mathrm{d}S \\
&\quad + \sum_{T \in \mathcal{T}_h} \sum_{j=1}^{n-1} \sum_{k=1}^{n-1} \int_{\partial T} \dfrac{\partial^3 w_h}{\partial \nu \partial \tau_j \partial \tau_k}\dfrac{\partial^2 v_h}{\partial \tau_j \partial \tau_k} \,\mathrm{d}S := E_{1,\nu\nu} + E_{1,\tau\nu} + E_{1,\tau\tau},
\end{aligned}
\end{equation}
where $\{\tau_j\}_{j=1}^{n-1}$ is the set of unit orthogonal vectors along $\partial T$.

\underline{Estimate of $E_1$}. For the Morley-type element, Lemma \ref{lm:Morley-tt} (tangential-tangential weak continuity for Morley) and Lemma \ref{lm:Morley-nn} (normal-normal weak continuity for Morley) imply that, by a standard scaling argument, 
$$ 
|E_{1,\nu\nu}| + |E_{1,\tau\tau}| \lesssim h|w_h|_{4,\Omega} |v_h|_{3,h}.
$$ 
For the Adini-type element, the $C^0$-continuity of $V_h$ and Lemma \ref{lm:Adini-nn} (normal-normal strong continuity for Adini) imply that $E_{1,\nu\nu} = E_{1,\tau\tau} = 0$. 

For the tangential-normal term, on each $(n-1)$-dimensional face of $T \in \mathcal{T}_h$, we notice that $\nu_i \nu_j = 0$ for $i \neq j$. It follows that $\frac{\partial v_h}{\partial x_j}$ is the tangent derivative along the faces on which $\nu_i$ is not zero. Therefore, 
$$ 
E_{1,\tau\nu} = 2 \sum_{T \in \mathcal{T}_h} \sum_{j=1}^{n-1}\int_{\partial T} \dfrac{\partial^3 w_h}{\partial \nu^2 \partial \tau_j}\dfrac{\partial^2 v_h}{\partial \nu \partial \tau_j} \,\mathrm{d}S
= 2 \sum_{T\in \mathcal{T}_h} \sum_{i=1}^n \sum_{j=1, j\neq i}^n \int_{\partial T} \dfrac{\partial^3 w_h}{\partial x_i^2 \partial x_j} \dfrac{\partial^2 v_h}{\partial x_i \partial x_j} \nu_i \,\mathrm{d}S.
$$ 
Then, we apply Lemma \ref{lm:err-tn} (estimate of tangential-normal terms) to conclude that
\begin{equation*}
|E_{1,\tau\nu}| \lesssim h \vert w_h \vert_{4,\Omega} \vert v_h \vert_{3,h}.
\end{equation*} 
By using interpolation of spaces and Lemma \ref{lemma: conforming-relative-H4} (approximation property of $H^4$ conforming relative), we have 
\begin{equation}\label{eq:estimate-E1}
|E_1 | \lesssim h^s\vert w_h \vert_{3+s,\Omega} \vert v_h \vert_{3,h} \lesssim h^s\vert u \vert_{3+s,\Omega} \vert v_h \vert_{3,h}. 
\end{equation}

\underline{Estimate of $E_2$}. Using the orthogonal projection $P_T^0: L^2(T) \to \mathcal{P}_0(T)$, we have 
$$
E_2 = -\sum_{T\in\mathcal{T}_h} \int_{T} \nabla (\nabla \Delta w_h - P_T^0 \nabla \Delta u) : \nabla^2 (v_h - \Pi_h^c v_h)\,\mathrm{d}x.
$$
Therefore, the inverse inequality and the standard approximation property of $P_T^0$ imply
\begin{equation} \label{eq:estimate-E2}
\begin{aligned}
|E_2| &\lesssim \sum_{T\in\mathcal{T}_h}  h_T^{-1} \| \nabla \Delta w_h - P_T^0 \nabla \Delta u \|_{0,T} \vert v_h - \Pi_h^c v_h\vert_{2,T} \\ 
& \lesssim \vert u-w_h \vert_{3,h} \vert v_h \vert_{3,h} + \sum_{T \in \mathcal{T}_h} \| \nabla \Delta u- P_T^0 \nabla \Delta u \|_{0, T} \vert v_h \vert_{3,T} \\
& \lesssim \left( |u - w_h|_{3,h} + h^s |u|_{3+s,\Omega} \right) |v_h|_{3,h}. 
\end{aligned}
\end{equation}

Combining \eqref{eq:consistency-1}, \eqref{eq:consistency-2}, \eqref{eq:estimate-E1}, \eqref{eq:estimate-E2} with the approximation property \eqref{err-conforming-relative-H4}, we prove the desired estimate. $\blacksquare$
\end{proof}

Based on the well-known Strang's Lemma
\begin{equation*}
\vert u-u_h\vert_{3,h} \lesssim \mathop{\inf}_{v_h \in V_{h0}} \vert u-v_h\vert_{3,h} + \mathop{\sup}_{0\neq v_h \in V_{h0}} \dfrac{\vert a_h(u,v_h) - (f,v_h) \vert}{\vert v_h \vert_{3,h}},
\end{equation*}
and the interpolation theory, we finally arrive at the following convergence result.  

\begin{theorem}
Let $V_{h0}$ be the finite element space of the $n$-rectangle Morley-type element or the $n$-rectangle Adini-type element. If $u \in H^{3+s}(\Omega) \cap H_0^3(\Omega)$ for $s \in [0,1]$ solves \eqref{eq:tri-harmonic} with $f \in L^2(\Omega)$, then 
\begin{equation}
\|u-u_h\|_{3,h} \lesssim h^s\vert u \vert_{3+s,\Omega} + h^3 \|f\|_{0,\Omega}.
\end{equation}
\end{theorem}

\section{Numerical Experiments} \label{sc:numerical}

In this section, we present several numerical results in both 2D and 3D to support the theoretical results.

\begin{example}[2D smooth solution]
In the first example, we test the Adini-type $H^3$-nonconforming finite element by solving the following two-dimensional triharmonic equation:
\begin{equation*}
(-\Delta)^3 u = f, ~ x \in \Omega,
\end{equation*} 
where $\Omega=(0,1)^2$. We choose the source term and  boundary conditions so that the exact solution is given by $u(x,y) = \cos(2 \pi x)\cos(2\pi y)$. We compute the numerical solution and calculate its convergence order in the sense of $H^k$ broken norm, where $k=1,2,3$. The following table shows the numerical results obtained on uniform $n$-rectangle meshes with various mesh-sizes $h$. We see that the numerical solution approximates to the exact solution with a linear convergence in the $H^3$ semi-norm, which corresponds with our theoretical prediction. Moreover, the table also indicates that both $\vert u-u_h \vert_{1,h}$ and $\vert u-u_h \vert_{2,h}$ is of the second-order.
\begin{table}[H]
\centering
\begin{tabular}{|c|c|c|c|c|c|c|c|c|}
\hline
$N$ & $\Vert u-u_h \Vert_0$ & order & $\vert u - u_h\vert_{1,h}$ & order & $\vert u - u_h \vert_{2,h}$ & order & $\vert u - u_h \vert_{3,h}$ & order \\ \hline
4 & 1.142e-01 & - & 7.092e-01 & - & 8.272e+00 & - & 1.436e+02 & -\\ \hline 
8 & 3.140e-02 & 1.86 & 1.822e-01 & 1.96 & 2.115e+00 & 1.97 & 6.971e+01 & 1.04\\ \hline 
16 & 7.997e-03 & 1.97 & 4.566e-02 & 2.00 & 5.320e-01 & 1.99 & 3.455e+01 & 1.01\\ \hline 
32 & 2.008e-03 & 1.99 & 1.142e-02 & 2.00 & 1.332e-01 & 2.00 & 1.723e+01 & 1.00\\ \hline 
64 & 5.027e-04 & 2.00 & 2.855e-03 & 2.00 & 3.331e-02 & 2.00 & 8.612e+00 & 1.00\\ \hline 
\end{tabular}
\caption{Numerical errors and observed convergence orders of Adini-type element for \textbf{Example 6.1}.}
\end{table}
\end{example}

\begin{example}[2D singular solution]
In this example, we solve the triharmonic equation on a two-dimensional L-shaped domain $\Omega = (-1,1)^2\setminus [0,1)\times(-1,0]$, in  which the solution has partial regularity. The exact solution is given in the polar coordinates $(r,\theta)$ as
\begin{equation*}
u(r,\theta) = r^{2.5} \sin(2.5 \theta).
\end{equation*}
Due to the singularity at the origin, we have $u \in H^{3+1/2-\epsilon} (\Omega)$ for any $\epsilon > 0$. Our method converges with the optimal rate $1/2$ in the $H^3$ broken norm, which is shown in the following table.
\begin{table}[H]
\centering
\begin{tabular}{|c|c|c|c|c|c|c|c|c|}
\hline
$N$ & $\Vert u-u_h \Vert_0$ & order & $\vert u - u_h\vert_{1,h}$ & order & $\vert u - u_h \vert_{2,h}$ & order & $\vert u - u_h \vert_{3,h}$ & order \\ \hline
2 & 4.031e-03 & - & 2.223e-02 & - & 2.049e-01 & - & 2.353e+00 & -\\ \hline 
4 & 1.589e-03 & 1.34 & 8.677e-03 & 1.36 & 8.988e-02 & 1.19 & 1.630e+00 & 0.53\\ \hline 
8 & 7.368e-04 & 1.11 & 4.002e-03 & 1.12 & 3.980e-02 & 1.18 & 1.140e+00 & 0.52\\ \hline 
16 & 3.442e-04 & 1.10 & 1.860e-03 & 1.11 & 1.776e-02 & 1.16 & 8.030e-01 & 0.51\\ \hline 
32 & 1.603e-04 & 1.10 & 8.571e-04 & 1.12 & 7.969e-03 & 1.16 & 5.670e-01 & 0.50\\ \hline 
64 & 7.474e-05 & 1.10 & 3.940e-04 & 1.12 & 3.594e-03 & 1.15 & 4.007e-01 & 0.50\\ \hline
\end{tabular}
\caption{Numerical errors on the L-shaped domain and observed convergence orders of Adini-type element for \textbf{Example 6.2}.}
\end{table}
\end{example}

\begin{example}[3D smooth solution]
For the last example, let us consider solving the trihamonic equation on a three-dimensional domain $\Omega=(0,1)^3$. We choose the right hand side function and appropriate boundary conditions so that the exact solution of \eqref{eq:tri-harmonic} is 
\begin{equation*}
u(x,y,z) = \sin(2\pi x)\cos(\pi y)\cos(\pi z).
\end{equation*}
We solve the equation using both Adini-type and Morley-type nonconforming element and the results are shown in Table~\ref{table-3d-Adini} and Table~\ref{table-3d-Morley}, respectively. We observe that both the finite element methods have a first-order convergence to the exact solution in $H^3$ norm. 
\begin{table}[H]
\centering
\begin{tabular}{|c|c|c|c|c|c|c|c|c|}
\hline
$N$ & $\Vert u-u_h \Vert_0$ & order & $\vert u - u_h\vert_{1,h}$ & order & $\vert u - u_h \vert_{2,h}$ & order & $\vert u - u_h \vert_{3,h}$ & order \\ \hline
2 & 8.721e-02 & - & 9.877e-01 & - & 1.008e+01 & - & 9.809e+01 & -\\ \hline 
4 & 6.866e-03 & 3.67 & 1.275e-01 & 2.95 & 2.302e+00 & 2.13 & 3.741e+01 & 1.39\\ \hline 
8 & 4.389e-04 & 3.97 & 1.702e-02 & 2.90 & 5.926e-01 & 1.96 & 1.781e+01 & 1.07\\ \hline 
16 & 5.028e-05 & 3.13 & 2.237e-03 & 2.93 & 1.494e-01 & 1.99 & 8.785e+00 & 1.02\\ \hline 
32 & 1.352e-05 & 1.89 & 3.181e-04 & 2.81 & 3.742e-02 & 2.00 & 4.377e+00 & 1.01\\ \hline 
\end{tabular}
\caption{Numerical errors and observed convergence orders of Adini-type element for \textbf{Example 6.3}.}
\label{table-3d-Adini}
\end{table}

\begin{table}[H]
\centering
\begin{tabular}{|c|c|c|c|c|c|c|c|c|}
\hline
$N$ & $\Vert u-u_h \Vert_0$ & order & $\vert u - u_h\vert_{1,h}$ & order & $\vert u - u_h \vert_{2,h}$ & order & $\vert u - u_h \vert_{3,h}$ & order \\ \hline
2 & 1.210e-01 & - & 1.216e+00 & - & 1.120e+01 & - & 1.153e+02 & - \\ \hline
4 & 9.100e-03 & 3.73 & 1.439e-01 & 3.08 & 2.473e+00 & 2.18 & 4.254e+01 & 1.44 \\ \hline
8 & 1.100e-03 & 3.05 & 1.990e-02 & 2.85 & 6.352e-01 & 1.96 & 1.888e+01 & 1.17\\ \hline
16 & 1.741e-04 & 2.66 & 2.900e-03 & 2.78 & 1.583e-01 & 2.00 & 8.949e+00 & 1.08\\ \hline
32 & 3.678e-05 & 2.24 & 5.192e-04 & 2.48 & 3.950e-02 & 2.00 & 4.401e+00 & 1.02\\ \hline
\end{tabular}
\caption{Numerical errors and observed convergence orders of Morley-type element for \textbf{Example 6.3}.}
\label{table-3d-Morley}
\end{table}
\end{example}

\section{Concluding Remarks} \label{sc:conclusion}
We propose two new families of nonconforming finite element for solving the sixth-order equations. We begin by proving some basic properties of such finite elements and discussing their approximation abilities in any dimensionaliy $n\geq 2$. 
After showing the approximation property and the stability of the interpolation operator,  we provide some key lemmas to obtain the main convergence theory for solving the sixth-order equations. By using the technique of conforming relatives, we discover that the numerical solutions of these on-conforming finite elements have an $h^s$ convergence order where $s\in[0,1]$, provided that the exact solution has $H^{3+s}$ regularity. We then give two examples to examine our theories for the cases $n=2$ and $n=3$ respectively, and another one example to show the robustness of our method when solving the triharmonic equation with a singular solution.

Although the new technique (i.e., exchange of sub-rectangles) presented in this paper focuses on the sixth-order equations, we believe it has the potential to be extended to higher-order equations. This will also be our future work.

\end{document}